\documentclass[11pt,a4paper]{article}

\usepackage{amsfonts, amsmath, amssymb, amsthm, enumerate}
\usepackage{psfrag,graphicx}
\newtheorem{proposition}{Proposition}[section]
\newtheorem{theorem}[proposition]{Theorem}
\newtheorem{remark}[proposition]{Remark}
\newtheorem{lemma}[proposition]{Lemma}

\newenvironment{proofof}[1]{\smallskip\noindent\emph{Proof of #1.}%
\hspace{1pt}}{\hspace{-5pt}{\nobreak\quad\nobreak\hfill\nobreak%
$\square$\vspace{8pt}\par}\smallskip\goodbreak}

\newcommand\sqr[2]{{\vcenter{\vbox{\hrule height.#2pt
   \hbox{\vrule width.#2pt height#1pt \kern#1pt
      \vrule width.#2pt}
   \hrule height.#2pt}}}}

\newcommand{\ds}{\displaystyle}

\renewcommand{\epsilon}{\varepsilon}
\newcommand{\eps}{\epsilon}

\newcommand{\BV}{\mathbf{BV}}

\newcommand{\C}[1]{{{\mathbf{C}^\mathbf{#1}}}}

\renewcommand{\L}[1]{{{\mathbf{L}^\mathbf{#1}}}}
\newcommand{\Lloc}[1]{\mathbf{L}^{\mathbf{1}}_{\mathrm{loc}}}

\newcommand{\reali}{{\mathbb{R}}}

\newcommand{\naturali}{\mathbb{N}}

\newcommand{\tv}{\mathrm{TV}}

\makeatletter

\newlength{\captionwidth}
\long\def\@makecaption#1#2{%
   \vskip 10\p@
   \setbox\@tempboxa\hbox{\small #1: #2}%
   \ifdim \wd\@tempboxa > \captionwidth %\hsize
       \hbox to\hsize{\hfil
       \parbox[t]{\captionwidth}{
        \small #1: #2\par}
       \hfil}
     \else
       \hbox to\hsize{\hfil\box\@tempboxa\hfil}%
   \fi}

\makeatother
\def\knp{K_{np}}

\title{Global weak solutions for a model of two-phase flow \\ with a single interface}

\author{}

\author{Debora Amadori\footnote{Department of Engineering and Computer Science and Mathematics, University of L'Aquila,
Italy}
\and Paolo Baiti\footnote{Department of Mathematics and Computer Science, University of Udine, Italy}
\and Andrea Corli\footnote{Department of Mathematics and Computer Science,
University of Ferrara, Italy}
\and Edda Dal Santo\footnotemark[2]
}
\begin{document}

\maketitle

\abstract{We consider a simple nonlinear hyperbolic system modeling the flow of an inviscid fluid. The model includes as state variable the mass density
fraction of the vapor in the fluid and then phase transitions can be taken into consideration; moreover, phase interfaces are contact
discontinuities for the system. We focus on the special case of initial data consisting of two different phases separated by an interface.
We find explicit bounds on the (possibly large) initial data in order that weak entropic solutions exist for all times.
The proof exploits a carefully tailored version of the front tracking scheme.}
%\smallskip

%\begin{AMS}
%\noindent\textit{2000~Mathematics Subject Classification:}
 % 35L65, 35L60, 35L67, 76T30\end{AMS}

%\smallskip

%\begin{keywords}
%\noindent\textit{Key words and phrases:}
%  Hyperbolic systems of conservation laws, phase transitions
%\end{keywords}

%%%%%%%%%%%%%%%%%%%%%%%%%%%%%%%%%%%%%%%%%%%%%%%%%%%%%%%%%%%%%%%%%
\section{Introduction}\label{sec:intro}

We consider the following nonlinear model for the one-dimensional flow of an inviscid fluid, where different phases can coexist:
\begin{equation}\label{eq:system}
\left\{
\begin{array}{ll}
v_t - u_x &= 0\,,\\
u_t + p(v,\lambda)_x &= 0\,,
\\
\lambda_t &= 0\,.
\end{array}
\right.
\end{equation}
Here $t>0$ and $x\in\reali$; moreover, $v>0$ is the specific
volume, $u$ the velocity and $\lambda$ the mass-density fraction of
vapor in the fluid. Then, we have $\lambda\in[0,1]$ and $\lambda=0$
characterizes the liquid phase while $\lambda=1$ the vapor phase. The pressure $p$ is given by
\begin{equation}\label{eq:pressure}
p(v,\lambda)= \frac{a^2(\lambda)}v\,,
\end{equation}
where $a$ is a $\C{1}$ function defined on $[0,1]$ and satisfying $a(\lambda)>0$ for every $\lambda\in[0,1]$.
We denote $U=(v,u,\lambda)\in\Omega\doteq(0,+\infty) \times \reali \times [0,1]$.

System \eqref{eq:system} is the homogeneous case of a more general model that was first introduced in \cite{Fan}. If $\lambda$ is constant,
then \eqref{eq:system} reduces to the isothermal $p$-system, where the global existence of weak solutions holds for initial data with arbitrary
total variation \cite{Nishida68,AmadoriGuerra01}. The global existence of weak solutions to the initial value problem for \eqref{eq:system} was
proved in \cite{amadori-corli-siam} under a suitable condition on the total variation of the initial data and the assumption $a'>0$;
a different proof of an analogous result has been recently provided in \cite{Asakura-Corli}.
The condition on the initial data was also stated in a slightly different way in \cite{amadori-corli-source} and requires, roughly speaking, that the
total variation of both pressure and velocity is suitably bounded by the total variation of $\lambda$; then, it reminds of the famous condition
introduced in \cite{NishidaSmoller} for the system of isentropic gasdynamics. A Glimm scheme to solve \eqref{eq:system} was proposed in
\cite{Peng}; we refer to \cite{amadori-corli-glimm} for a short proof of the Glimm estimates, which improve those given in \cite{Peng}.

We refer to \cite{GST2007} for the extension of Nishida's result to the initial-value problem in Special Relativity,
and to \cite{LiuIB}, \cite{Liu} to the problem of large solutions to nonisentropic gas dynamics.

A model analogous to \eqref{eq:system} is also studied in  \cite{HoldenRisebroSande, HoldenRisebroSande2}, where the pressure is
$v^{-\gamma}$ and the state variable $\lambda$ is replaced by the adiabatic exponent $\gamma>1$; also in this case the global existence of
solutions is proved under a condition that has the same flavor of that discussed above. At last, we refer to \cite{Dafermos} for a comprehensive
discussion of the problem of the global existence of solutions for systems of conservation laws.

In this paper we focus on a %very
particular class of initial data for \eqref{eq:system}: the state variable $\lambda$ is constant both for $x<0$ and
for $x>0$. More precisely, for $x\in\reali$ we consider initial data
\begin{equation}
U(x,0)= U_o(x) = \left(v_o(x), u_o(x), \lambda_o(x) \right),\label{init-data}
\end{equation}
where
\begin{equation}\label{eq:lambda_Riemann}
\lambda_o(x) = \left\{
\begin{array}{ll}
\lambda_\ell & \hbox{ if }x<0\,,\\
\lambda_r & \hbox{ if }x>0\,,\\
\end{array}
\right.
\end{equation}
for two constant values $\lambda_\ell\ne\lambda_r\in[0,1]$. Phase interfaces are stationary in model \eqref{eq:system}; then, the assumption
\eqref{eq:lambda_Riemann} reduces the study of the initial value problem for \eqref{eq:system} to that of two initial value problems for two
isothermal $p$-systems, which are coupled through the interface at $x=0$. In other words, the flow remains in the two phases characterized by
the values $\lambda_\ell$ and $\lambda_r$ as long as a solution exists.
The case of initial data giving rise to {\em two} phase interfaces is addressed in a forthcoming paper \cite{ABCD2}.

The problem we are dealing with can be understood in a different way as follows. Phase interfaces are contact discontinuities for
system \eqref{eq:system}; then, in a sense, we fall into the general framework of the perturbation of a Riemann solution.
For this subject we refer to \cite{BressanColombo-Unique, CorliSableTougeronS, CorliSableTougeronC, Schochet-Glimm, Chern},
where however the perturbation is small in the $\BV$ norm. In our case the perturbation leaves unchanged the initial datum for $\lambda$
but it is not necessarily small in the other state variables. % \footnote{I due problemi sono diversi, infatti in genere la perturbazione tende a 0 all'infinito, per come l'hai scritto dunque dovremmo guardare i casi in cui i valori $u(\pm\infty)$ coincidono. Le nostre soluzioni si potrebbero vedere come perturbazioni di PdR completi (come e' scritto nel paragrafo), in cui i dati sono i valori a $\pm\infty di (v,u,\lambda)(0)$. Nel nostro caso, il dato iniziale per una discontinuita' di contatto e' $(\bar v,\bar u,\bar \lambda)$ che vale $(v_\ell,\bar u,\lambda_\ell)$ se $x<0$ e $(v_r,\bar u,\lambda_r)$ se $x>0$, con $v_\ell$ e $v_r$ opportuni; una sua pertubazione, nel senso di \cite{BressanColombo-Unique, CorliSableTougeronS, CorliSableTougeronC, Schochet-Glimm} ha dati iniziali $(v_o(x),u_o(x),\lambda_o(x))=(\bar v,\bar u,\bar\lambda)+ (\tilde v(x),\tilde u(x),\tilde\lambda(x))$. Essa e' del tipo dei nostri dati iniziali solo se $\tilde\lambda=0$. Viceversa, un dato iniziale come il nostro e' una perturbazione di una discontinuita' di contatto se $\tilde\lambda=0$ e $v_\ell$, $v_r$ sono opportuni.}
The problem of a {\em small} perturbation of a Riemann solution and the related existence of globally defined solutions was thoroughly studied
in \cite{LewickaBV}; in \cite{AmCo12-Chambery}  the conditions given in \cite{LewickaBV} are made explicit for system \eqref{eq:system}.
%However, we emphasize that the general results of \cite{LewickaBV} apply to {\em small} perturbations of a Riemann solution while here, thanks to the special form of system \eqref{eq:system}, the initial data are not required to be necessarily small.

The main result of this paper is stated in Theorem \ref{thm:main} and concerns the global existence of weak solutions to the initial value problem \eqref{eq:system}, \eqref{init-data}, \eqref{eq:lambda_Riemann}, provided that \eqref{eq:pressure} holds and the initial data satisfy suitable conditions. The focus is precisely on weakening as much as possible such conditions, allowing for large initial data: the result in \cite{amadori-corli-siam} mentioned above clearly applies to the present situation, but it is here greatly improved.

The proof of Theorem \ref{thm:main} follows the same steps as Theorem $2.2$ in \cite{amadori-corli-siam}. However, several novelties have been introduced here:
\begin{itemize}
\item[-] a Glimm functional that better accounts for nonlinear interactions with the phase wave;
\item[-] refined interaction estimates on the amplitude of the reflected waves (Lemma~\ref{lem:shock-riflesso});
\item[-] an original treatment of \emph{non-physical} waves in the front tracking algorithm;
\item[-] a simpler proof of the decay of the reflected waves at a geometric rate, as the number of reflections increases
(see Proposition~\ref{prop:alpha} and Remark~\ref{rem:65}).
\end{itemize}

\noindent In particular, as a consequence of this new approach, we require no conditions on the maximal amplitude of the
phase wave, differently from \cite[(2.8)]{amadori-corli-siam} and the equivalent formulation in \cite[(3.6)]{amadori-corli-source}.

In spite of the fact that initial data \eqref{eq:lambda_Riemann} seem to reduce system \eqref{eq:system} to two systems of two conservation
laws, we cannot avoid the introduction of non-physical waves \cite{Bressanbook} in the scheme, as a formal example
in \cite{AmCo08-Proceed-Maryland} shows. Nevertheless, we can let all these non-physical waves propagate along the same vertical front
carrying the contact discontinuity, in order to give an immediate bound on the number of fronts: this represents a remarkable algorithmic
advantage and the main feature of the front tracking used here. On the other hand, we recall that if $\lambda$ is constant then non-physical
waves need not to be introduced, see \cite{AmadoriGuerra01, BaitiDalSanto}.

The plan of the paper is the following. %At last, here follows the plan of the paper.
In Section \ref{main} we state our main result, % and make a comparison with that proved in \cite{amadori-corli-siam}.
while in Section \ref{prelim} we first provide some information on the Riemann problem and then
show how to treat non-physical waves by introducing a composite wave together with the phase wave.
Consequently, we introduce two solvers to be used in the front-tracking scheme, which shows up in Section \ref{sec:app_sol}.
Section \ref{sec:interactions} deals with interactions while in the last Section \ref{sec:Cauchy} we prove the convergence and consistency of the algorithm and make a comparison with the result in \cite{amadori-corli-siam}. In a final short appendix we show how the damping coefficient $c$ introduced in \eqref{eq:chi_def}, which plays a key role in the paper, is also fundamental in the stability analysis of Riemann problems in the sense of \cite{Schochet-Glimm}.

%%%%%%%%%%%%%%%%%%%%%%%%%%%%%%%%%%%%%%%%%%%%%%%%%%%%%%%%%%%%%%%%%

\section{Main Result}\label{main}
\setcounter{equation}{0}

In this section we state our existence theorem. First, we define $a_r=a(\lambda_r)$,  $a_\ell=a(\lambda_\ell)$ and
\begin{equation}\label{def:A_0}
  \delta_2 = 2\, \frac{a_r-a_\ell}{a_r+a_\ell}\,.
 \end{equation}
Notice that $\delta_2$ ranges over $(-2,2)$ as soon as $a_r$, $a_\ell$ range over $\reali_+$.
The quantity $\delta_2$ measures the strength of the contact discontinuity located at $x=0$ as in
\cite{AmCo06-Proceed-Lyon, Peng} and it does not change by interactions with waves of the other families, see Lemma \ref{lem:interazioni}.
%Another important quantity
%introduced in \cite{amadori-corli-siam} is the damping coefficient $d(m)$, where $m>0$ denotes the maximum strength of the interacting waves;
%see Lemma \ref{lem:shock-riflesso}. The function $d(m)$ is positive and strictly
%increasing, valued in $[0,1)$ and satisfies $d(0)=0$, see Figure~\ref{fig:dm}.
%Roughly speaking, the function $d(m)$ represents a damping coefficient related to the shocks reflected after an interaction of two waves of the same family.

%At last
We denote $p_o(x) \doteq p\left(v_o(x), \lambda_o(x) \right)$.

\begin{theorem}\label{thm:main}
Assume \eqref{eq:pressure} and consider initial data \eqref{init-data}, \eqref{eq:lambda_Riemann} with $v_o(x)\ge \underline{v}>0$,
for some constant $\underline{v}$. Let $\delta_2$ be as in \eqref{def:A_0}.

There exists a strictly decreasing function $K$ defined for $r\in(0,2)$ and satisfying
\begin{equation}\label{rage}
\lim_{r\to 0+} K(r) = +\infty\,,\qquad  \lim_{r\to 2-} K(r) = \frac 29\log\left(2+\sqrt 3\right),
\end{equation}
such that,
if $\delta_2\not = 0$ and the initial data satisfy
\begin{equation}\label{hyp2}
  \tv\left(\log(p_o)\right) \,+\, \frac{1}{\min\{a_r,a_\ell\}}  \tv(u_o) < K(|\delta_2|)\,,
\end{equation}
then the Cauchy problem \eqref{eq:system}, \eqref{init-data} has a
weak entropic solution $(v,u,\lambda)$ defined for
$t\in\left[0,+\infty\right)$. If $\delta_2=0$ the same conclusion holds with $K(|\delta_2|)$ replaced by $+\infty$ in \eqref{hyp2}.

Moreover, the solution is valued in a compact set of $\Omega$ and there is a constant $C=C(\delta_2)$ such that
for every $t\in\left[0,+\infty\right)$ we have
\begin{equation}\label{eq:tv-est}
\tv \left(v(t,\cdot), u(t,\cdot)\right)  \le  C\,.
\end{equation}
\end{theorem}

We refer to \eqref{eq:K-explicit} for the definition of the function $K$; therefore, condition \eqref{hyp2} is {\em explicit}.
We recall that related results of global existence of solutions with large data \cite{NishidaSmoller, LiuIB, Liu, HoldenRisebroSande, HoldenRisebroSande2} do not precise the threshold of smallness of the initial data.

Moreover, we observe that condition \eqref{hyp2} is trivially satisfied if
\begin{equation}\label{eq:smalldata}
\tv\left(\log(p_o)\right) \,+\, \frac{1}{\min\{a_r,a_\ell\}}  \tv(u_o)\le \frac 29\log\left(2+\sqrt 3\right),
\end{equation}
because of \eqref{rage}. Then, problem \eqref{eq:system}, \eqref{init-data} has a global solution if \eqref{eq:smalldata} is satisfied and $v_o(x)\ge \underline{v}>0$ holds. This is a striking difference with respect to the results in \cite{amadori-corli-siam, amadori-corli-source}, where the corresponding bound in the right-hand side vanishes at a critical threshold.
Moreover, Theorem \ref{thm:main} improves the main result in \cite{amadori-corli-siam}, when restricted to the case of a single contact discontinuity; we refer to Subsection~\ref{subsec:proof} for a comparison.
At last, we point out that if $\delta_2=0$ we recover the result of \cite{Nishida68}.
% \footnote{Qui, se vogliamo, potremmo mettere l'esempio di PE che il risultato vale per la somma e non per il massimo delle due TV.}

% The present theorem improves the main result in \cite{amadori-corli-siam},
% when referred to the case of the single contact discontinuity.
% See the end of Subsection~\ref{subsec:proof}.

It is left open the question of whether the global existence of solutions to \eqref{eq:system}, \eqref{init-data} for any $\BV$ initial data $v_o$, $u_o$ occurs,
opposite to the possibility of the blow-up in finite time for certain $\BV$ data.

%%%%%%%%%%%%%%%%%%%%%%%%%%%%%%%%%%%%%%%%%%%%%%%%%%%%%%%%%%%%%%%%%

\section{The Riemann problem and the composite wave}\label{prelim}
\setcounter{equation}{0}

In this section we first briefly recall some basic facts about system \eqref{eq:system}, its wave curves and the solution to the Riemann problem; we refer to \cite{AmCo06-Proceed-Lyon,amadori-corli-siam} for more details. Next, we introduce a composite wave which sums up the effects of the contact discontinuity and the non-physical waves. We then show two Riemann solvers that make use of the composite wave.

Under assumption (\ref{eq:pressure}) system (\ref{eq:system}) is strictly hyperbolic in $\Omega$ with eigenvalues $e_1 = -\sqrt{-p_v(v,\lambda)}$, $e_2 = 0$, $e_3 = \sqrt{-p_v(v,\lambda)}$; the eigenvalues $e_1$ and $e_3$ are genuinely nonlinear while $e_2$ is linearly degenerate.

For $i=1,3$, the right shock-rarefaction curves through the point $U_o = (v_o,u_o,\lambda_o)$ for (\ref{eq:system}) are
\begin{equation}
v \mapsto \left(v,u_o + 2a(\lambda_o) h(\eps_i),\lambda_o\right)\,,\qquad v>0\,, \ i=1,3\,,
\label{eq:lax13}
\end{equation}
where the strength $\eps_i$ of an $i$-wave is defined as
\begin{equation}\label{eq:strengths}
\eps_1=\frac{1}{2}\log\left(\frac{v}{v_o}\right),
\quad \eps_3=\frac{1}{2}\log\left(\frac{v_o}{v}\right)
\end{equation}
and the function $h$ is defined by
\begin{equation}\label{h}
h(\eps)= \begin{cases}
\eps& \mbox{ if } \eps \ge 0\,,\\
\sinh \eps& \mbox{ if } \eps < 0\,.
\end{cases}
\end{equation}

Then, rarefaction waves have positive strengths and shock waves have negative strengths.
The wave curve through $U_o$ for $i=2$ is defined by
\begin{equation*}
\lambda\mapsto\left(v_o\ds\frac{a^2(\lambda)}{a^2(\lambda_o)},
u_o,\lambda\right)\,,\qquad \lambda\in[0,1]\,.
\end{equation*}
Then, the pressure is constant along a $2$-curve; the strength of a $2$-wave is defined by
\begin{equation*}
\eps_2 = 2\, \frac{a(\lambda)-a(\lambda_o)}{a(\lambda)+a(\lambda_o)}\,.
\end{equation*}
Now, we consider the Riemann problem for (\ref{eq:system}) with initial condition
\begin{equation}\label{eq:incond}
(v,u,\lambda)(0,x)=\left\{
\begin{array}{ll}
(v_\ell,u_\ell,\lambda_\ell)=U_\ell & \hbox{ if }x<0\,,
\\
(v_r,u_r,\lambda_r)=U_r & \hbox{ if }x>0\,,
\end{array}
\right.
\end{equation}
for $U_\ell$ and $U_r$ in $\Omega$. We write %$a_r=a(\lambda_r)$,
$p_r= a^2_r/v_r$, % and similarly for % $a_\ell$,
$p_\ell= a^2_\ell/v_\ell$. %; at last, we define

\begin{proposition}[\cite{AmCo06-Proceed-Lyon}]\label{prop:RP}
The Riemann problem \eqref{eq:system}, \eqref{eq:incond} has a
unique $\Omega$-valued solution in the class of solutions
consisting of simple Lax waves, for any pair of states $U_\ell$, $U_r$ in $\Omega$.

Moreover, if $\eps_i$ is the strength of the $i$-wave, $i=1,2,3$, then
\begin{equation}\label{eq:stimaRiemann}
\eps_3-\eps_1  = \frac{1}{2}\log\left(\frac{p_r}{p_\ell}\right)\,,
\qquad 2\left(a_\ell h(\eps_1) +a_r h(\eps_3)\right)   =
u_r-u_\ell\,,
\end{equation}
\[
% \eps_2 = 2\, \frac{a(\lambda_{r})-a(\lambda_{\ell})}{a(\lambda_{r})+a(\lambda_\ell)}.
\eps_2 = 2\, \frac{a_{r}-a_{\ell}}{a_{r}+a_\ell}\,.
\]
\end{proposition}

The proof of Theorem \ref{thm:main} relies on a wave-front tracking algorithm that introduces non-physical waves \cite{Bressanbook}, which, however, are only needed to solve some Riemann problems involving interactions with the $2$-wave, see Section \ref{sec:app_sol}. Following \cite{amadori-corli-siam}, two states $U_\ell$ and $U_r$ as in \eqref{eq:incond} can be connected by a non-physical wave if $v_\ell=v_r$ and $\lambda_\ell=\lambda_r$; the strength of a non-physical wave is defined as
\begin{equation}\label{eq:strengthnp}
\delta_0=u_r-u_\ell\,.
\end{equation}
Then, a non-physical wave changes neither the side values of $v$ nor those of $\lambda$, while a $2$-wave does not change the side values of $u$. This suggests to define a new wave by composing the $2$-wave with a non-physical wave, with the condition that we assign {\em zero speed} to non-physical waves and we locate them at $x=0$. The order of composition does not matter, because a $2$-wave and a non-physical wave act on different state variables. This procedure differs from the one used in \cite{amadori-corli-siam}. Then, we define the {\em composite $(2,0)$-wave curve} through a point $U_o=(v_o,u_o,\lambda_\ell)$ by
\begin{equation}\label{eq:composite-wave}
u\mapsto \left(  (a_r^2/a_\ell^2) v_o, u,\lambda_r \right)
\end{equation}
and its strength by
\begin{equation*}
\delta_{2,0} = u - u_o\,.
\end{equation*}
The above definition of strength is motivated by the fact that the quantity $\delta_2$ remains constant at any interaction with $1$- or $3$-waves \cite{amadori-corli-siam}. Clearly, a $(2,0)$-wave reduces to the $2$-wave as long as non-physical waves are missing. At last, we notice that the pressure does not change across a $(2,0)$-wave.

In this way, we are left to deal with waves of family $1$, $3$ and a single composite $(2,0)$-wave, which is no more entropic. A Riemann solver analogous to that provided in Proposition \ref{prop:RP} is needed; however, since we have a single contact discontinuity $\delta_2$ and we are going to use the Riemann solver only to solve interactions, we state the following result into such a form.

\begin{proposition}[Pseudo Accurate Solver]\label{prop:PS-AS}
Consider the interaction at time $t$ of a $\delta_{2,0}$-wave with an $i$-wave of strength $\delta_i$, $i=1,3$. Then the Riemann problem at time $t$ has a unique $\Omega$-valued solution, which is formed by waves $\eps_1$, $\delta_{2,0}$, $\eps_3$, where $\eps_1$, $\eps_3$ belong to the first and the third family, respectively. Moreover, we have
\begin{equation}\label{eq:stimaRiemannAS}
\eps_3-\eps_1  = \frac{1}{2}\log\left(\frac{p_r}{p_\ell}\right),
\qquad 2\left(a_\ell h(\eps_1) +a_r h(\eps_3)\right)   =
u_r-u_\ell - \delta_{2,0}\,.
\end{equation}
\end{proposition}

%%%%%%%%%%%%%%%%%%%%%%%% Figure interactions 20ASL

\begin{figure}[htbp]
\begin{picture}(100,100)(-80,-15)
\setlength{\unitlength}{1pt}

%figure (a)

\put(70,0){

\put(0,0){\line(0,1){40}}
\put(-2,-5){\makebox(0,0){$\delta_{2,0}$}}
\put(30,10){\line(-1,1){30}}
\put(40,5){\makebox(0,0){$\delta_1$}}
\put(0,40){\line(0,1){30}}
\put(2,75){\makebox(0,0){$\delta_{2,0}$}}
\put(0,40){\line(-3,2){30}}
\put(-32,67){\makebox(0,0){$\eps_1$}}
\put(0,40){\line(1,1){30}}
\put(32,77){\makebox(0,0){$\eps_3$}}
\put(-20,40){\makebox(0,0){$U_\ell$}}
\put(20,40){\makebox(0,0){$U_r$}}
\put(20,0){\makebox(0,0){$U_m$}}
\put(-10,60){\makebox(0,0){$U_p'$}}
\put(9,60){\makebox(0,0){$U_q'$}}

%figure name
\put(0,-20){\makebox(0,0){$(a)$}} }
%end figure (a)

%figure (b)
\put(230,0){

\put(0,0){\line(0,1){40}}
\put(-2,-5){\makebox(0,0){$\delta_{2}$}}
\put(30,10){\line(-1,1){30}}
\put(40,5){\makebox(0,0){$\delta_1$}}
\put(0,40){\line(0,1){30}}
\put(2,75){\makebox(0,0){$\delta_{2}$}}
\put(0,40){\line(-3,2){30}}
\put(-32,67){\makebox(0,0){$\eps_1$}}
\put(0,40){\line(1,1){30}}
\put(32,77){\makebox(0,0){$\eps_3$}}
\put(-20,40){\makebox(0,0){$V_\ell'$}}
\put(20,40){\makebox(0,0){$U_r$}}
\put(20,0){\makebox(0,0){$U_m$}}
\put(-10,60){\makebox(0,0){$V_p'$}}
\put(9,60){\makebox(0,0){$V_q'$}}

%figure name
\put(0,-20){\makebox(0,0){$(b)$}} }
%end figure (b)

\end{picture}

\caption{\label{fig:inter20ASL}{$(a)$: interaction with the $(2,0)$-wave solved with the Pseudo Accurate solver, case $i=1$; $(b)$: the auxiliary problem.}}
\end{figure}

%%%%%%%%%%%%%%%%%%%%%%%%%%%%% End Figure interactions 20ASL

\begin{proof} We only consider the case $i=1$ and refer to Figure \ref{fig:inter20ASL}; the other case is analogous.
Consider the auxiliary problem in Figure \ref{fig:inter20ASL}({\em b}), where $V_\ell' = U_\ell + (0,\delta_{2,0},0)$. We simply shifted the left state in order to be able to solve the interaction as if it was with an actual $2$-wave. Indeed, by Proposition \ref{prop:RP} %\footnote{Not completely true: now rarefactions are cut...}
we uniquely find $\eps_1$, $\eps_3$ and states $V_p'$, $V_q'$ such that \eqref{eq:stimaRiemannAS} holds. Then, the interaction in Figure \ref{fig:inter20ASL}({\em a}) is solved by the same waves $\eps_1$, $\eps_3$ and by states $U_p' = V_p'-(0,\delta_{2,0},0)$, $U_q' = V_q'$. Finally, \eqref{eq:stimaRiemannAS} holds by construction.

Notice that we get the same result by shifting the other two states at the right. Indeed, consider the auxiliary problem in Figure \ref{fig:inter20ASR}({\em b}), where $V_r''= U_r - (0,\delta_{2,0},0)$ and $V_m'' = U_m - (0,\delta_{2,0},0)$.

%%%%%%%%%%%%%%%%%%%%%%%% Figure interactions 20ASR

\begin{figure}[htbp]
\begin{picture}(100,100)(-80,-15)
\setlength{\unitlength}{1pt}

%figure (a)

\put(70,0){

\put(0,0){\line(0,1){40}}
\put(-2,-5){\makebox(0,0){$\delta_{2,0}$}}
\put(30,10){\line(-1,1){30}}
\put(40,5){\makebox(0,0){$\delta_1$}}
\put(0,40){\line(0,1){30}}
\put(2,75){\makebox(0,0){$\delta_{2,0}$}}
\put(0,40){\line(-3,2){30}}
\put(-32,67){\makebox(0,0){$\eps_1$}}
\put(0,40){\line(1,1){30}}
\put(32,77){\makebox(0,0){$\eps_3$}}
\put(-20,40){\makebox(0,0){$U_\ell$}}
\put(20,40){\makebox(0,0){$U_r$}}
\put(20,0){\makebox(0,0){$U_m$}}
\put(-10,60){\makebox(0,0){$U_p''$}}
\put(9,60){\makebox(0,0){$U_q''$}}

%figure name
\put(0,-20){\makebox(0,0){$(a)$}} }
%end figure (a)

%figure (b)
\put(230,0){

\put(0,0){\line(0,1){40}}
\put(-2,-5){\makebox(0,0){$\delta_{2}$}}
\put(30,10){\line(-1,1){30}}
\put(40,5){\makebox(0,0){$\delta_1$}}
\put(0,40){\line(0,1){30}}
\put(2,75){\makebox(0,0){$\delta_{2}$}}
\put(0,40){\line(-3,2){30}}
\put(-32,67){\makebox(0,0){$\eps_1$}}
\put(0,40){\line(1,1){30}}
\put(32,77){\makebox(0,0){$\eps_3$}}
\put(-20,40){\makebox(0,0){$U_\ell$}}
\put(20,40){\makebox(0,0){$V_r''$}}
\put(20,0){\makebox(0,0){$V_m''$}}
\put(-10,60){\makebox(0,0){$V_p''$}}
\put(9,60){\makebox(0,0){$V_q''$}}

%figure name
\put(0,-20){\makebox(0,0){$(b)$}} }
%end figure (b)

\end{picture}

\caption{\label{fig:inter20ASR}{$(a)$: interaction with the $(2,0)$-wave solved with the Pseudo Accurate solver, case $i=1$; $(b)$: the auxiliary problem.}}
\end{figure}

%%%%%%%%%%%%%%%%%%%%%%%%%%%%% End Figure interactions 20A%%%

By Proposition \ref{prop:RP} we uniquely find $\eps_1$, $\eps_3$ (the same as before, since $u_{r}''-u_{l}=u_{r}-u_{l}'=u_{r}-u_{l}-\delta_{2,0}$) and states $V_p''$, $V_q''$. Then, the interaction in Figure \ref{fig:inter20ASR}({\em a}) is solved by the same waves $\eps_1$, $\eps_3$ and by states $U_p'' = V_p''$ and $U_q'' = V_q''+(0,\delta_{2,0},0)$. It is then straightforward to check that $U_p' = U_p''$ and $U_q' = U_q''$.
\end{proof}

Another solver is used below. We introduce it in the same framework of Proposition \ref{prop:RP}.

\begin{proposition}[Pseudo Simplified Solver]\label{prop:PS-SS}
Consider the interaction at time $t$ of a $\delta_{2,0}$-wave with an $i$-wave of strength $\delta_i$, $i=1,3$.
Then the Riemann problem at time $t$ can be solved by an $i$-wave of the same strength $\delta_i$ and a unique wave $\eps_{2,0}$, where
\begin{equation}\label{eq:inter20}
\eps_{2,0} = \left\{
\begin{array}{ll}
\delta_{2,0} + 2(a_r-a_\ell)h(\delta_1)& \hbox{ if $i=1$\,,}
\\
\delta_{2,0} - 2(a_r-a_\ell)h(\delta_3)& \hbox{ if $i=3$\,.}
\end{array}
\right.
\end{equation}
\end{proposition}

%%%%%%%%%%%%%%%%%%%%%%%% Figure interactions 20

\begin{figure}[htbp]
\begin{picture}(100,100)(-80,-15)
\setlength{\unitlength}{1pt}

%figure (a)

\put(70,0){

\put(0,0){\line(0,1){40}}
\put(-2,-5){\makebox(0,0){$\delta_{2,0}$}}
\put(30,10){\line(-1,1){30}}
\put(40,5){\makebox(0,0){$\delta_1$}}
\put(0,40){\line(0,1){30}}
\put(2,75){\makebox(0,0){$\eps_{2,0}$}}
\put(0,40){\line(-1,1){30}}
\put(-32,77){\makebox(0,0){$\delta_1$}}
\put(-20,40){\makebox(0,0){$U_\ell$}}
\put(20,40){\makebox(0,0){$U_r$}}
\put(20,0){\makebox(0,0){$U_m$}}
\put(-8,60){\makebox(0,0){$U_q$}}

%figure name
\put(0,-20){\makebox(0,0){$(a)$}} }
%end figure (a)

%figure (b)
\put(230,0){

\put(0,0){\line(0,1){40}}
\put(-2,-5){\makebox(0,0){$\delta_{2,0}$}}

\put(-30,10){\line(1,1){30}}
\put(-33,3){\makebox(0,0){$\delta_3$}}

\put(0,40){\line(0,1){30}}
\put(2,75){\makebox(0,0){$\eps_{2,0}$}}

\put(0,40){\line(1,1){30}}
\put(35,75){\makebox(0,0){$\delta_3$}}

\put(-20,40){\makebox(0,0){$U_\ell$}}
\put(30,40){\makebox(0,0){$U_r$}}
\put(-10,10){\makebox(0,0){$U_m$}}
\put(10,60){\makebox(0,0){$U_q$}}

%figure name
\put(0,-20){\makebox(0,0){$(b)$}} }
%end figure (b)

\end{picture}

\caption{\label{fig:inter20}{Interactions with the $(2,0)$-wave solved with the Pseudo Simplified solver. $(a)$: from the right ($i=1$);
$(b)$: from the left ($i=3$).}}
\end{figure}
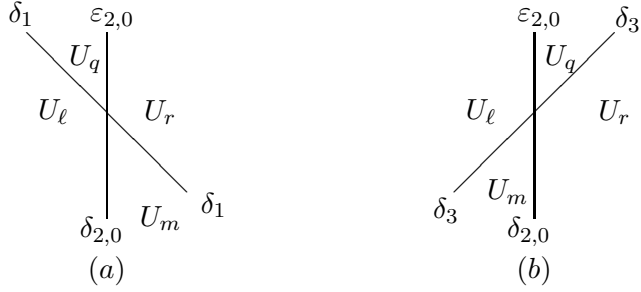

%%%%%%%%%%%%%%%%%%%%%%%%%%%%% End Figure interactions 20

\begin{proof} We refer to Figure \ref{fig:inter20}. We recall that, by \cite[Lemma 2]{AmCo06-Proceed-Lyon},
the commutation of a $1$-wave (or a $3$-wave) with the $2$-wave $\delta_2$ only modifies the $u$ component.

In the case when a $1$-wave interacts, it is easy to check that $u_q = u_\ell + 2a_\ell h(\delta_1)$ and $u_m=u_\ell + \delta_{2,0}$;
then, we compute $\eps_{2,0}$ by $u_\ell + 2a_\ell h(\delta_1) +\eps_{2,0} = u_\ell + \delta_{2,0} + 2a_rh(\delta_1)$.
The other case is analogous.
\end{proof}

%%%%%%%%%%%%%%%%%%%%%%%%%%%%%%%%%%%%%%%%%%%%%%%%%%%%%%%%%%%%%%%%%

\section{Approximate solutions}\label{sec:app_sol}
\setcounter{equation}{0}

We use Propositions \ref{prop:PS-AS} and \ref{prop:PS-SS} to build up the piecewise-constant approximate solutions to (\ref{eq:system}) that are needed for the wave-front tracking scheme \cite{Bressanbook,amadori-corli-siam}. We first approximate the initial data \eqref{init-data}: for any $\nu\in\naturali$ we take a sequence $(v^\nu_o,u^\nu_o)$ of piecewise constant functions with a finite number of jumps such that, denoting $p^\nu_o=a^2(\lambda_o)/v^\nu_o$,

\begin{enumerate}
\item $\tv \log(p^\nu_o)\leq \tv \log(p_o)$, $\tv u^\nu_o\leq \tv u_o$;

\item $\lim_{x\to-\infty} (v^\nu_o,u^\nu_o)(x)
  =\lim_{x\to-\infty} (v_o,u_o)(x)$;

\item $\|(v^\nu_o,u^\nu_o) - (v_o,u_o)\|_{\L1}\leq 1 /\nu$.
\end{enumerate}

We introduce two strictly positive parameters: $\eta=\eta_\nu$, that controls the size of rarefactions, and a threshold $\rho=\rho_\nu$, that determines which of the two Pseudo Riemann solver is to be used. Here follows a description of the scheme that improves the algorithm of  \cite{amadori-corli-siam} and adapts it to the current situation.

\begin{enumerate}[{\em (i)}]

\item At time $t=0$ we solve the Riemann problems at each point of jump of 
$(v^\nu_o, u^\nu_o, \lambda_o)(\cdot, 0+)$ as follows: shocks are not modified 
while rarefactions are approximated by fans of waves, each of them having size 
less than $\eta$. More precisely, a rarefaction of size $\eps$ is approximated 
by $N=[\eps/\eta]+1$ waves whose size is $\eps/N<\eta$; we set their speeds to 
be equal to the characteristic speed of the state at the right. 
Then $(v,u,\lambda)(\cdot,t)$ is defined until some wave fronts interact; by slightly changing the speed of some waves we can assume that only \emph{two} fronts interact at a time.

\item When two wave fronts of the families $1$ or $3$ interact, we solve the Riemann problem at the interaction point. If one of the incoming waves is a rarefaction, after the interaction it is prolonged (if it still exists) as a single discontinuity with speed equal to the characteristic speed of the state at the right.  If a new rarefaction is generated, we employ the Riemann solver described in step {\em (i)} and split the rarefaction into a fan of waves having size less than $\eta$.

\item When a wave front of family $1$ or $3$ with strength $\delta$ interacts with the composite wave at a time $t>0$, we proceed as follows:
    \begin{itemize}
    \item if $|\delta|\ge\rho$, we use the {\em Pseudo Accurate solver} introduced in 
    Proposition \ref{prop:PS-AS}, partitioning the possibly new rarefaction
    according to \emph{(i)};
    \item if $|\delta|<\rho$, we use the {\em Pseudo Simplified solver} of Proposition \ref{prop:PS-SS}.
    \end{itemize}

\end{enumerate}

%%%%%%%%%%%%%%%%%%%%%%%%%%%%%%%%%%%%%%%%%%%%%%%%%%%%%%%%%%%%%%%%%

\section{Interactions}\label{sec:interactions}
\setcounter{equation}{0}

In this section we analyze the interactions of waves. If $\delta_2=0$, i.e. if $a(\lambda_\ell)=a(\lambda_r)$, then the initial data \eqref{init-data} reduce \eqref{eq:system} to a $p$-system where the pressure $p$ only depends on $v$. The results of \cite{AmadoriGuerra01,amadori-corli-siam} apply and we recover the famous result of \cite{Nishida68}. Then, we assume from now on that $\delta_2\ne0$. For simplicity, we focus on the case
\[
a(\lambda_\ell) < a(\lambda_r)\,.
\]
As a consequence we have $\delta_2>0$; the other case is entirely similar.

For $t>0$ at which no interactions occur, and for $\xi\ge1$, $\knp>0$, $K\ge 1$ to be determined, we introduce the functionals
\begin{align}
  L & =  \sum_{\genfrac{}{}{0pt}{}{i=1,3}{\gamma_i>0}}|\gamma_i| +
  \xi\sum_{\genfrac{}{}{0pt}{}{i=1,3}{\gamma_i<0}}|\gamma_i|
  +\knp |\gamma_{2,0}|\,,\label{L-xi}
  \\[2mm]
  \nonumber
  V & =  \sum_{\genfrac{}{}{0pt}{}{i=1,3}{\gamma_i>0,\,\mathcal{A}}}|\gamma_i| +
  \xi\sum_{\genfrac{}{}{0pt}{}{i=1,3}{\gamma_i<0,\,\mathcal{A}}}|\gamma_i|\,, \quad\quad  Q = \delta_2 V,
\\[2mm]
\label{F}
F & =  L+ K\, Q\,.%|\delta_2| V\,.
\end{align}
By $\gamma_{2,0}$ we mean the strength of the composite wave. The summation in $V$ is performed only over the set $\mathcal{A}$ of waves {\em approaching} the
 front carrying the composite wave, namely the waves of the family $1$ (and $3$) located at the right (left, respectively) of $x=0$.
 The term  $Q$ %%$|\delta_2| V \doteq Q$
 is then the \lq\lq usual\rq\rq\ quadratic interaction potential due to the contact discontinuity at $x=0$. We also introduce
\begin{equation*}
\bar{L}  =  \sum_{i=1,3}|\gamma_i| = \frac 12 \tv(\log p(t,\cdot))\,.
\end{equation*}

\begin{remark}\rm
%\footnote{Check this remark.}
The functional defined in \eqref{F} differs from \cite[(5.1)]{amadori-corli-siam} because of the presence of the parameter $\xi$ in $V$ and, consequently, in the interaction potential $Q$, leading to better estimates and a more general result. %An extension of the functional to the case of a more general function $\lambda_o$ does not appear possible.
An extension of the functional to the case of a more general function
$\lambda_o$ appears possible, however with some extra condition on $\lambda_o$.
The specific case with only two phase interfaces is addressed in \cite{ABCD2}.
\end{remark}

Under the notation of Figure \ref{fig:Glimmone}, we shall make use of the identities \cite{AmCo06-Proceed-Lyon,Peng}
\begin{align}
\eps_3 -  \eps_1 & =   \alpha_3 + \beta_3 -\alpha_1 - \beta_1\,,
\label{tre-uno}
\\
a_\ell h(\eps_1) + a_r h(\eps_3) & =   a_\ell h(\alpha_1) + a_m
h(\alpha_3) + a_m h(\beta_1) + a_r h(\beta_3)\,. \label{tre-due}
\end{align}

%%%%%%%%%%%%%%%%% FIGURE GLIMMONE %%%%%%%%%%%%%%%

\begin{figure}[htbp]
\begin{picture}(100,80)(-120,0)
\setlength{\unitlength}{1pt}

\put(100,0){

%left
\put(-40,0){\line(1,2){20}} \put(-40,0){\line(0,1){40}}
\put(-40,0){\line(-1,1){40}}
\put(-62,30){\makebox(0,0){$\alpha_1$}}
\put(-47,32){\makebox(0,0){$\alpha_2$}}
\put(-18,30){\makebox(0,0){$\alpha_3$}}

%right
\put(60,0){\line(1,2){20}}
\put(60,0){\line(0,1){40}}
\put(60,0){\line(-3,5){24}}
\put(37,30){\makebox(0,0){$\beta_1$}}
\put(55,33){\makebox(0,0){$\beta_2$}}
\put(83,30){\makebox(0,0){$\beta_3$}}

%above
\put(10,40){\line(2,3){26}} \put(10,40){\line(0,1){40}}
\put(10,40){\line(-3,5){24}} \put(-13,70){\makebox(0,0){$\eps_1$}}
\put(17,73){\makebox(0,0){$\eps_2$}}
\put(38,70){\makebox(0,0){$\eps_3$}}

\put(-80,10){\makebox(0,0){$U_\ell$}}
\put(10,10){\makebox(0,0){$U_m$}}
\put(100,10){\makebox(0,0){$U_r$}}

}
\end{picture}
\caption{\label{fig:Glimmone}{A general interaction pattern.}}
\end{figure}
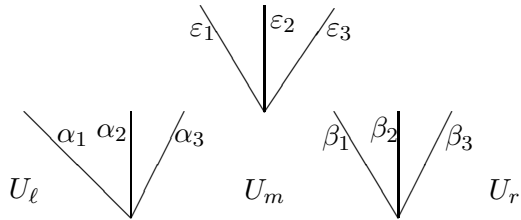

%%%%%%%%%%%%%%%%% END FIGURE GLIMMONE %%%%%%%%%%%%%%%

%%%%%%%%%%%%%%%%%%%%%%%%%%%%%%%%%%%%%%%%%%%%%%%%%%%%%%%%%%%%%%%%%

\subsection{Interactions with the composite wave}\label{subsec:Interactions20wave}

We first consider the interactions of a $1$- or $3$-wave with a $(2,0)$-wave. As in \cite{AmCo06-Proceed-Lyon}, we notice that they give rise to the following pattern of solutions:
\begin{equation}\label{eq:patterns}
\begin{array}{rclcrcl}
(2,0)\times 1R &\to& 1R+(2,0)+3R\,,&\quad &(2,0)\times 1S & \to & 1S+(2,0)+3S\,,
\\
3R\times (2,0) & \to& 1S+(2,0)+3R\,,
&\quad &3S\times (2,0) &\to& 1R+(2,0)+3S\,.
\end{array}
\end{equation}
In the following we often assume that, for some fixed $m>0$, any interacting $i$-wave, $i=1,3$, with strength $\delta_i$ satisfies
\begin{equation}\label{rogna}
|\delta_i|\le m\,.
\end{equation}
We usually denote with $\delta_k$ (and $\eps_k$) the interacting waves (respectively, the waves produced by the interaction).

\begin{lemma}\label{lem:interazioni}
Assume that a wave $\delta_i$, $i=1,3$, interacts with a $\delta_{2,0}$-wave.

If the Riemann problem is solved by the Pseudo Accurate solver, then the strengths $\eps_i$ of the outgoing waves satisfy $\eps_{2,0}=\delta_{2,0}$ and
\begin{align}
|\eps_i - \delta_i| = |\eps_j| & \le   \displaystyle
\frac{1}{2}\,\delta_2  |\delta_i|,\,\quad
i,j=1,3,\, i\ne j\,,
\label{eq:33Pengnew}
\\
|\eps_1|+|\eps_3| & \leq  \left\{
\begin{array}{ll}
|\delta_1| + \delta_2|\delta_1|&\qquad\hbox{ if $i=1$\,,}
\\
|\delta_3| &\qquad\hbox{ if $i=3$\,.}
\end{array}
\right.
\label{eq:stima-interazione-semplice}
\end{align}
If the Riemann problem is solved by the Pseudo Simplified procedure and we assume \eqref{rogna}, then there exists $C_o=C_o(m)$ such that
\begin{align}
|\eps_{2,0}-\delta_{2,0}|& \le  \displaystyle
C_o\, \delta_2 |\delta_i|\,.
\label{eq:stima-interazione-composta}
\end{align}
\end{lemma}

\begin{proof}
The estimates \eqref{eq:33Pengnew} and \eqref{eq:stima-interazione-semplice} easily follow from Proposition \ref{prop:PS-AS} and are carried out as  in \cite{AmCo06-Proceed-Lyon}.
The proof of the second part relies on the estimates of \cite[Proposition 5.12]{amadori-corli-siam}; we have
\begin{equation*}
|\eps_{2,0}-\delta_{2,0}| = 2|a_{r}-a_\ell|\,|h(\delta_i)|\leq 2 a_r \frac{\sinh m}m \delta_2\,
|\delta_i|\,,
\end{equation*}
whence (\ref{eq:stima-interazione-composta}) immediately follows once we set $C_o(m) \doteq 2a_{r} \sinh m/m$.
\end{proof}

\begin{proposition}\label{Delta-F-2wave}
Assume that a wave $\delta_i$, $i=1,3$, interacts with a $\delta_{2,0}$-wave at time $t$.

\noindent In the cases where the Pseudo Accurate procedure is used, then $\Delta F(t) < 0$ if
\begin{equation}\label{Kappa-mu}
  K > \max\left\{\frac{\xi-1}{2}\,,\,{1}\right\}\,.
\end{equation}
In the cases where the Pseudo Simplified procedure is used, then $\Delta F(t)< 0$ if
\begin{equation}\label{eq:knp}
\knp < \frac{K}{C_o}\,.
\end{equation}
\end{proposition}

\begin{proof}
We first consider the case where the Pseudo Accurate solver is used and use the notation of Figure~\ref{fig:inter20ASL}. By \eqref{tre-uno} and Lemma \ref{lem:interazioni}, we have
\begin{equation*}
\left\{
\begin{array}{lll}
%%i=1:\quad &
\eps_1-\delta_1 = \eps_3, &\quad |\eps_1|-|\delta_1| = |\eps_3|\,,  &\qquad \mbox{if  }i=1 \\[1mm]
\eps_1+\delta_3 = \eps_3, &\quad |\delta_3|-|\eps_1|  =  |\eps_3|\,, &\qquad \mbox{if  }i=3\,.
\end{array}
\right.
\end{equation*}
\smallskip\noindent{\fbox{$i=1$.}} If the interacting wave is a rarefaction, then $\Delta L = 2|\eps_3|\le \delta_2|\delta_1|$ and
$\Delta V =  - |\delta_1|$. Therefore, by \eqref{Kappa-mu} we deduce
\begin{equation}\label{eq:DeltaF2}
\Delta F = \Delta L + K \, \delta_2 \Delta V
\le \left\{ 1 - K \right\}\delta_2|\delta_1|  < 0\,.
\end{equation}
If the interacting wave is a shock, we have the same estimates with $\xi$ as a factor.

\smallskip\noindent{\fbox{$i=3$.}}
%In this case $\Delta \bar{L} =0$.
If the interacting wave is a shock, then $\Delta L = |\eps_1|+\xi|\eps_3|-\xi|\delta_3|=-(\xi-1)|\epsilon_1|\le0$, $\Delta V =-\xi|\delta_3|<0$ and
\begin{equation}\label{eq:2bs}
\Delta F = -(\xi-1)|\epsilon_1| -K\delta_2\xi|\delta_3| \le -K\delta_2\xi|\delta_3| <0\,.
\end{equation}
If the %interacting
wave is a rarefaction, then $\Delta L = \xi|\eps_1| + |\eps_3| - |\delta_3| = (\xi -1 )|\eps_1| \le (\xi -1)\, \delta_2|\delta_3|/2$ and
$\Delta V \,=\, - |\delta_3|$. By \eqref{Kappa-mu} we obtain again
\begin{equation}\label{eq:2Fr}
  \Delta F = \Delta L + K \, \delta_2\, \Delta V
\le  \delta_2|\delta_3|  \left\{ \frac{\xi -1}2 - K \right\} < 0\,.
\end{equation}

If the Pseudo Simplified solver is used, then $\Delta V \le -|\delta_i|$ ($i=1,3$) and $\Delta L =\knp|\eps_{2,0}|-\knp|\delta_{2,0}| \leq \knp
C_o\delta_2|\delta_i|$ by \eqref{eq:stima-interazione-composta}. Hence, by \eqref{eq:knp} we get
\begin{equation*}
\Delta F \leq \delta_2|\delta_i| (\knp C_o - K)<0\,.
\end{equation*}
\end{proof}

%%%%%%%%%%%%%%%%%%%%%%%%%%%%%%%%%%%%%%%%%%%%%%%%%%%%%%%%%%%%%%%%%

\subsection{Interactions between $1$- and $3$-waves}

In this subsection we analyze the interactions between $1$- and $3$-waves, see Figure \ref{fig:inter3133}.
%%%%%%%%%%%%%%%%%%%%%%%% inter31 %%%%%%%%%%%%%%%%%%%%%%%%%%%%%%

\begin{figure}[htbp]
\begin{picture}(100,80)(-130,-15)
\setlength{\unitlength}{0.8pt}

%Figure a
\put(-20,0){
\put(0,40){\line(-2,-3){30}}\put(-40,0){\makebox(0,0){$\delta_3$}}
\put(20,0){\line(-1,2){20}} \put(30,0){\makebox(0,0){$\delta_1$}}
\put(0,40){\line(1,2){15}} \put(20,75){\makebox(0,0){$\eps_3$}}
\put(0,40){\line(-1,1){30}} \put(-30,75){\makebox(0,0){$\eps_1$}}
\put(0,-20){\makebox(0,0){$(i)$}}
%end figure (a)
}

%Figure b1
\put(140,0){
\put(0,40){\line(1,-1){45}}
\put(15,0){\makebox(0,0){$\alpha_1$}}
\put(0,40){\line(2,-3){30}}
\put(55,0){\makebox(0,0){$\beta_1$}}
\put(0,40){\line(1,2){15}} \put(20,75){\makebox(0,0){$\eps_3$}}
\put(0,40){\line(-1,1){30}} \put(-30,75){\makebox(0,0){$\eps_1$}}
}
%
%Figureb2
\put(300,0){ \put(0,40){\line(-1,-1){45}}
\put(-55,0){\makebox(0,0){$\alpha_3$}}
\put(0,40){\line(-2,-3){30}}\put(-10,0){\makebox(0,0){$\beta_3$}}
\put(0,40){\line(1,2){15}} \put(20,75){\makebox(0,0){$\eps_3$}}
\put(0,40){\line(-1,1){30}} \put(-30,75){\makebox(0,0){$\eps_1$}}
\put(-80,-20){\makebox(0,0){$(ii)$}}
}
%end figure (b)

\end{picture}

\caption{\label{fig:inter3133}{Interactions of $1$- and $3$-waves.}}
\end{figure}
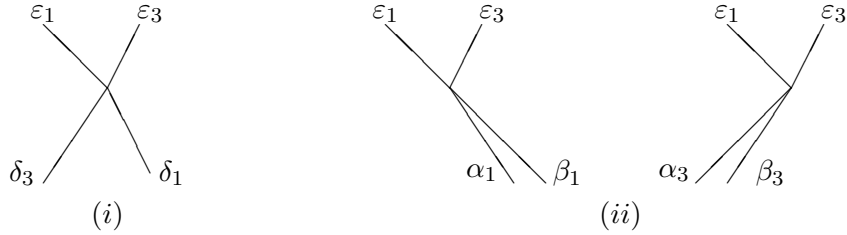

%%%%%%%%%%%%%%%%%%%%%%%%%%%%% inter 31
\begin{lemma}%[\cite{amadori-corli-siam}]
\label{lem:shock-riflesso}
For the interaction patterns in Figure \ref{fig:inter3133}, the following holds.

\begin{enumerate}[{(i)}]
\item Two interacting waves of different families
cross each other without changing strengths.

\item Let $\alpha_i$, $\beta_i$ be two interacting waves of the same family and $\eps_1$, $\eps_3$ the outgoing waves.

\begin{enumerate}[({ii}.a)]

\item  If both incoming waves are shocks, then the outgoing wave of the same family is a shock and satisfies $|\eps_i|>\max \{|\alpha_i|,|\beta_i|\}$; the reflected wave is a rarefaction.

\item  If the incoming waves have different signs, then the reflected wave is a shock; both the amounts of shocks and rarefactions of the $i$-family decrease across the interaction.  Moreover for $j\ne i $ and $\alpha_i<0<\beta_i$ one has
    \begin{align}\label{eq:chi_def}
    |\eps_j| & \le c(\alpha_i) \cdot \min\{|\alpha_i|,|\beta_i|\}\,,\qquad c(z) \doteq \frac{\cosh z -1}{\cosh z+1}\,.
    \end{align}
\end{enumerate}
\end{enumerate}
\end{lemma}

\begin{remark}\rm The inequality \eqref{eq:chi_def} generalizes the one stated in \cite[Lemma B.1]{amadori-corli-siam}
for the case $SR$, $RS\to SS$; moreover, in that case we provide below a simpler proof.

%Moreover we notice that the coefficient $c(m)$ coincides with the \cite{Schochet-Glimm}
%See \cite{AmCo12-Chambery} and...
\end{remark}

\begin{proofof}{Lemma \ref{lem:shock-riflesso}} We only need to prove \eqref{eq:chi_def}, the rest being already proved in \cite[Lemmas 5.4--5.6]{amadori-corli-siam}. For simplicity we assume $i=3$ and distinguish between two cases according to the outgoing wave $\eps_3$.
%Notice that, if this case occurs (that is, when a rarefaction emerges from the interaction), then the rarefaction must be sufficiently large to overcome the shock.
Indeed, we remark that there exists a function $x_o(\cdot)$ such that $\eps_3$ is a rarefaction iff $\beta_3\ge x_o(|\alpha_3|)$; see \cite[Lemma B.1]{amadori-corli-siam}. In the limiting case $\beta_3= x_o(|\alpha_3|)$ the shock and the rarefaction cancel each other and $\eps_3=0$; the interaction gives only rise to the reflected wave $\eps_1$. By setting $x=|\beta_{3}|$ and $z=|\alpha_3|$, from \eqref{tre-uno} and \eqref{tre-due} we find the equation valid for $\eps_3=0$, namely
$$
\sinh(x-z) -\sinh z +x=0\,,
$$
which implicitly defines the function $x=x_o(z)$.

\paragraph{\fbox{$SR,\, RS\to SR$}}\quad The starting point is to specialize \eqref{tre-uno} and \eqref{tre-due} to the present case:
%In questo caso, le equazioni per $\alpha_{3},\beta_{3},\eps_{1},\eps_{3}$ si riscrivono come
\begin{align}
    |\eps_1|+|\eps_3|&=-|\alpha_3|+|\beta_3|\,,\label{1}\\
    \sinh(|\eps_1|)-|\eps_3|&=\sinh(|\alpha_3|)-|\beta_3|\,.\label{2}
\end{align}
By summing up \eqref{1} and \eqref{2}  we find that
\begin{equation}
    \label{4}
    \sinh(|\eps_1|)+|\eps_1|=\sinh(|\alpha_3|)-|\alpha_3|\,.
\end{equation}
To prove \eqref{eq:chi_def} it is enough to prove that    %%%%Quello che vogliamo ottenere \`e una stima del tipo
\begin{equation}
    \label{3}
    |\eps_1|\le c(\alpha_3) |\alpha_3|\,.
\end{equation}
Indeed, from \eqref{1} we infer that $|\alpha_3|<|\beta_3|$ and therefore \eqref{3} implies \eqref{eq:chi_def}.

To prove \eqref{3}, we introduce the notation $|\eps_1|=y$ and $|\alpha_3|=z$, so that \eqref{4} rewrites as
$$
 G(y,z)\doteq\sinh y+y-\sinh z+z =0\,.
$$
By a simple application of the Implicit Function Theorem, there exists a function $y=y(z)\ge 0$, defined for all $z\ge0$, such that $G\left(y(z),z\right)=0$.

Since $G_{y}(y,z)=\cosh y +1>0$, in order to prove that $y(z)\le c(z)z$ it is enough to prove that $g(z)\doteq G(c(z)z,z)>0$, that is
%$G(0,z)<0$ for $z>0$ and $y\mapsto G(y,z)\to\infty$ as $y\to\infty$
\begin{equation}
    \label{stima-comune}
    g(z)=(c(z)+1)z+\sinh(c(z)z)-\sinh z>0\,.
\end{equation}
Using the fact that $c(z)z<z$, the Mean Value Theorem and the simple identity
\begin{equation*}%   \label{relaz-cm}
    1+c(z)=\left(1-c(z)\right)\cosh z\,,
\end{equation*}
we find that
$$
g(z)=\left(c(z)+1\right)z + \left(c(z)z -z\right) \cosh\zeta >  z\left[ c(z)+1+ \left(c(z) - 1\right) \cosh z \right] =0\,,
$$
for $c(z)z<\zeta<z$. Hence, we have proved \eqref{stima-comune}.

\paragraph{\fbox{$SR,\, RS\to SS$}}\quad Again, we start from \eqref{tre-uno} and \eqref{tre-due} that can now be rewritten as
%In questo caso, le equazioni per $\alpha_{3},\beta_{3},\eps_{1},\eps_{3}$ si riscrivono come
\begin{align}
    |\eps_1|-|\eps_3|&=-|\alpha_3|+|\beta_3|\,,\label{1-SS}\\
    \sinh(|\eps_1|)+\sinh(|\eps_3|)&=\sinh(|\alpha_3|)-|\beta_3|\,.\nonumber%\label{2-SS}
\end{align}
Set $x=|\beta_{3}|$, $y=|\eps_{1}|$,
$z=|\alpha_{3}|$ and define the function
$$
F(x,y;z)=\sinh y+\sinh(y-x+z)-\sinh z+x\,,
$$
which is subject to the constraints
$$
z\geq0,
\quad
0\leq x< x_{o}(z),
\quad
\max\{0,x-z\}<y<\min\{x,z\}\,.
$$
%Here $x_o(z)$ denotes the maximal amplitude of the rarefaction $\beta_3$ in order to produce an outgoing shock. More precisely,
%for $\beta_3=x_o(|\alpha_3|)$ the two incoming waves cancel each other, which means that the outcoming wave $\eps_3$ vanishes.
By the Implicit Function Theorem, there exists a function $y=y(x;z)$ such that $F\left(x,y(x;z);z\right)\equiv 0$. 
Moreover, by denoting with $y'$ the derivative of $y$ with respect to $x$ and so on, we have
\begin{gather*}
    y'=-\frac{F_{x}}{F_{y}}\,,\qquad
    y''=-\frac{F_{xx}+2F_{xy}y'+F_{yy}(y')^{2}}{F_{y}}\,,
\end{gather*}
where
\begin{gather*}
    F_{x}=1-\cosh(y-x+z) <0,
    \quad
    F_{y}=\cosh(y-x+z)+\cosh y >0\,,\\
    F_{xx}=-F_{xy}= \sinh(y-x+z)>0\,,\quad
%    F_{xy}=-\sinh(y-x+z)<0\,,\\
    F_{yy}=\sinh(y-x+z)+\sinh y>0\,.
\end{gather*}
Therefore $y'>0$ and
\begin{gather*}
    y''(x;z)=-\frac{\sinh\left(y-x+z\right)(1-y')^{2} + \sinh\left(y\right)(y')^{2}}{F_{y}} <0\,.
\end{gather*}
Hence $x\mapsto y(x;z)$ is concave down and thus
$$
y(x;z)\leq y'(0;z)x=c(z)x\,.
$$
To complete the proof of \eqref{eq:chi_def}, it remains to prove that $y(x;z)\leq c(z)z$. To do this, simply recall that $y'>0$ and then
$$
y(x;z)\leq y\left(x_o(z);z\right) \leq c(z)z\,,
$$
where the last inequality holds because it coincides with \eqref{3} in the limiting case $\beta_3=x_o(z)$, $z=|\alpha_3|$.
\end{proofof}

\begin{remark}{\rm
Under the notation of the proof of case {\em (ii.b)} in Lemma \ref{lem:shock-riflesso}, i.e., $x=\beta_i$, $z=|\alpha_i|$, we see that
the size of the reflected shock is
\begin{equation}\label{eq:eps_j-explicit}
|\eps_j| = \left\{
\begin{array}{ll}
y(x;z) & \hbox { if } x \le x_o(z)\,,
\\
y(z) & \hbox { if } x > x_o(z)\,.
\end{array}
\right.
\end{equation}
The strength $\eps_j$ is a continuous function of $x$ since $y\left(x_o(z);z\right) = y(z)$ for every $z$.
In particular, assume that $\beta_i> x_o(|\alpha_i|)$, so that $\eps_i$ is a rarefaction. For $\beta_i$ in this range, the size of $\eps_j$ does not change by \eqref{eq:eps_j-explicit} and the part of $\beta_i$ exceeding $x_o(|\alpha_i|)$ is entirely propagated along $\eps_i$. This holds since the interaction only affects that part of $\beta_i$ whose amplitude is exactly $x_o(|\alpha_i|)$. We refer to Figure \ref{fig:y(x;3)} for a graph of $|\eps_j|$ as a function of $\beta_i$.

We notice that this behavior of $\eps_j$ is mimicked by the damping coefficient $c$ in \eqref{eq:chi_def}, which only depends on the size of $\alpha_i$. }

%%%%%%%%%%%%%%%%%%%%%%%%%%%%%%%%%%%%

\begin{figure}[htbp]
\begin{center}
{\includegraphics[angle=90,width=10cm]{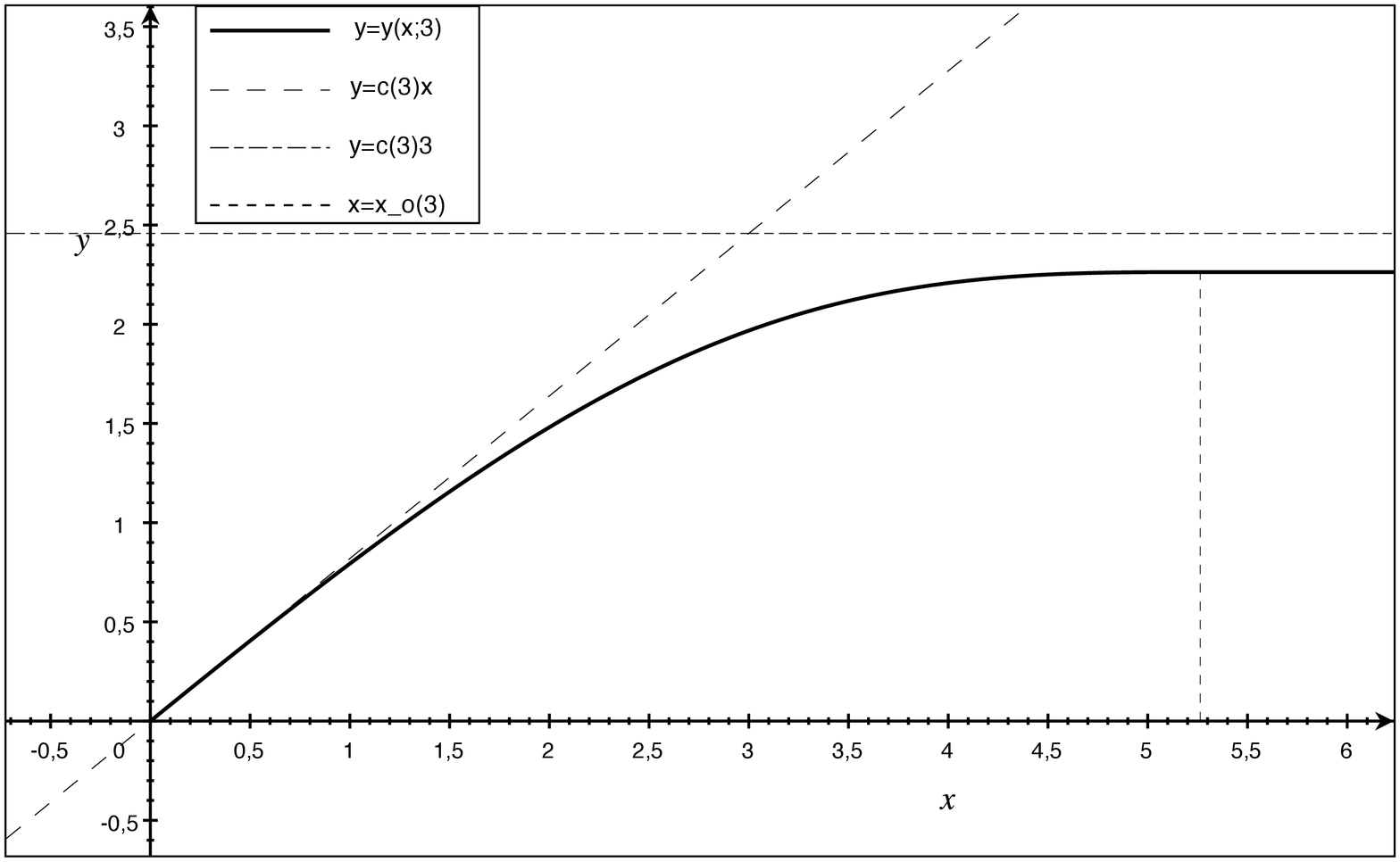}}
\end{center}
\vspace{-6mm}
\caption{The reflected shock in case {\em (ii.b)} of
Lemma~\ref{lem:shock-riflesso}.
The solid curve is the graph of $|\eps_j|=y$ as a function of $\beta_i=x$, for $|\alpha_i|=z=3$; see \eqref{eq:eps_j-explicit}. The vertical line marks the passage of $\eps_i$ from shock to rarefaction; on its right, $|\eps_j|$ assume the constant value $y=(x_o(z);z)$.
The two remaining dashed lines refer to the bounds in \eqref{eq:chi_def}; in particular, since $\lim_{z\to+\infty} \left(c(z)z-y(x_o(z);z) \right)=0$, the horizontal bound becomes asymptotically accurate.}
\label{fig:y(x;3)}
\end{figure}

%%%%%%%%%%%%%%%%%%%%%%%%%%%%%%%%%%%%

\end{remark}

\begin{remark}\label{rem-d(m)}\rm
In case {\em (ii.a)} of Lemma \ref{lem:shock-riflesso}, one can prove for the reflected rarefaction that
\begin{equation}\label{eq:eps_j-rar}
|\eps_j| \le d\left(\max\{|\alpha_i|,|\beta_i|\}\right) \min\left\{ |\alpha_i|,|\beta_i|\right\},
\end{equation}
for a suitable function $d(z)>c(z)$; see \cite[Lemma 5.6]{amadori-corli-siam}. Estimate \eqref{eq:eps_j-rar} is analogous to \eqref{eq:chi_def} but the damping coefficient $d\left(\max\{|\alpha_i|,|\beta_i|\}\right)$ cannot be replaced by $c\left(\max\{|\alpha_i|,|\beta_i|\}\right)$.
%does not hold with damping coefficient $c\left(\max\{|\alpha_i|,|\beta_i|\}\right)$.
This easily follows by a second order expansion of the function $\tau(a,b)$ in \cite[Lemma 5.6]{amadori-corli-siam} or simply by arguing as in the proof of case {\em (ii.b)}. However, we shall see in the following proposition that the decreasing of the functional $F$ only depends on the coefficient $c$ and not on $d$.
\end{remark}

\begin{proposition}\label{prop:DeltaF33}
Consider the interactions of two wave fronts of the same family $1$ or $3$, and assume \eqref{rogna}. Then $\Delta F \le 0$ if
\begin{equation}\label{eq:sogliazza}
1 < \xi \le  \frac{1}{c(m)}\quad\hbox{ and }\quad
 K \le \frac{\xi-1}{\delta_2}\,.
\end{equation}
\end{proposition}

\begin{proof}
The proof takes into account the possible wave configurations.
We use the notation of Lemma \ref{lem:shock-riflesso} and  assume $i=3$. %The identity \eqref{tre-uno} reduces to
%\begin{equation}\label{eq:3133}
%\eps_3-\eps_1 = \alpha_3+\beta_3\,.
%\end{equation}
%The proof takes into account the possible wave configurations; in all cases we prove

\paragraph{\fbox{$SS\to RS$}}\quad
We start by proving that
\begin{equation}\label{Delta_L_xi_13}
\Delta L + |\eps_1|(\xi - 1) =  0\,,
\end{equation}
that holds for all $\xi\ge 1$. Indeed, in this case one has $\Delta \bar{L}=0$ by \eqref{tre-uno} %{eq:3133}
and then
\begin{equation*}%\label{Delta_xi_L_xi_SSRS}
  \Delta L + (\xi -1) |\eps_1| = \xi (|\eps_1| +
  |\eps_3|-|\alpha_3|- |\beta_3|)= 0\,.
\end{equation*}
If $\Delta V>0$ then $\Delta V=|\eps_1|$; hence, by \eqref{eq:sogliazza} and \eqref{Delta_L_xi_13} we obtain
\begin{equation*}
  \Delta F
\le |\eps_1| \left\{ - (\xi-1) + K \delta_2\right\} \le 0\,.
\end{equation*}

\paragraph{\fbox{$SR,\, RS\to SR,\, SS$}}\quad Assume $\alpha_3<0<\beta_3$. We now
prove the stronger inequality
\begin{equation}\label{Delta_L_xi_13-SR}
\Delta L + |\eps_1|\xi (\xi - 1)  \le  0\,.
\end{equation}
%which implies (\ref{Delta_L_xi_13}).

If $\epsilon_3$ is a shock, then %\eqref{Delta_L_xi_13-SR} rewrites as
%$$\xi^2 |\eps_1| + \xi(|\eps_3|-|\alpha_3|) - |\beta_3|\le 0\,.$$
%To prove the inequality above,
we use \eqref{1-SS}%{eq:3133}
, \eqref{eq:chi_def} and (\ref{eq:sogliazza})${}_1$ to obtain
\begin{align*}
\Delta L + |\eps_1|\xi (\xi - 1) &=\xi^2 |\eps_1| + \xi(|\eps_3|-|\alpha_3|) - |\beta_3|\\
&= \xi^2 |\eps_1| + \xi(|\eps_1|-|\beta_3|) - |\beta_3|\\
&= (\xi+1) (\xi|\eps_1| - |\beta_3|) \le 0\,.
\end{align*}
Therefore \eqref{Delta_L_xi_13-SR} holds in this case.

On the other hand, if $\epsilon_3$ is a rarefaction, then the left hand side of
\eqref{Delta_L_xi_13-SR} turns out to be %$\xi^2 |\eps_1| + |\eps_3| -\xi|\alpha_3| - |\beta_3|$. Then it is bounded by
\begin{align*}%\label{Delta_xi_L_xi_SRSR}
\xi^2 |\eps_1| + |\eps_3| -\xi|\alpha_3| - |\beta_3|\,.
%&\le \xi^2 |\eps_1| + \xi(|\eps_3|-|\alpha_3|) - |\beta_3|\\
%&= (\xi+1) (\xi|\eps_1| - |\beta_3|) \le 0\,,
\end{align*}
From \eqref{1} %{eq:3133}
we have
$|\eps_3|< |\beta_3|$, while (\ref{eq:chi_def}) and (\ref{eq:sogliazza})${}_1$ imply $\xi |\eps_1|\le |\alpha_3|$. This completely  proves \eqref{Delta_L_xi_13-SR}.

If $\Delta V>0$, then $\Delta V=\xi|\eps_1|$ and hence
\begin{equation}
  \Delta F
  \le  \xi |\eps_1| \left\{ - (\xi-1) + K \delta_2\right\} \le 0
\end{equation}
by \eqref{eq:sogliazza}${}_2$.  This concludes the proof of the lemma.
\end{proof}

%\begin{remark}
%In the above proof we showed that in every case we have the estimate
%\[
%\xi\cdot \xi_{\epsilon_1}|\epsilon_1| + \xi_{\epsilon_3}|\epsilon_3|- \xi_{\alpha_3}|\alpha_3| - \xi_{\beta_3}|\beta_3| \le 0.
%\]
%\end{remark}

%%%%%%%%%%%%%%%%%%%%%%%%%%%%%%%%%%%%%%%%%%%

%\newpage
\subsection{Decreasing of the functional $F$ and control of the variations}
In order that $\Delta F\le 0$ at any interaction, we need $K$ to satisfy both \eqref{Kappa-mu} and $\eqref{eq:sogliazza}_2$:
\begin{equation}\label{cond_K_xi}
  \max\left\{\frac{\xi-1}{2}\,,\,{1}\right\}< K \le \frac{\xi-1}{\delta_2}\,. %%%% \le deve cambiare in < per avere \mu<1 !!
\end{equation}
This is possible if $1+\delta_2 < \xi$; hence, by $\eqref{eq:sogliazza}_1$ we require that $\xi$ satisfies
\begin{equation}\label{cond_xi_delta2}
  1+\delta_2 < \xi \le  \frac{1}{c(m)} %\frac{1}{\sqrt{d(m)}}
  \,.
\end{equation}
In turn, this is possible if
\begin{equation}\label{ip_m_delta2}
   c(m) <  \frac{1}{1+\delta_2} %\frac{1}{(1+\delta_2)^2}
   \,.
 \end{equation}
We notice that inequality \eqref{ip_m_delta2} is certainly satisfied if $c(m)\le 1/3$ because $\delta_2<2$.
Therefore, we choose the parameters $m$, $\xi$ and $K$ as follows:
\begin{enumerate}
\item We determine the maximum size $m$ of the waves in the approximate solution by assuming \eqref{ip_m_delta2}; we recall that $c$ is a strictly increasing function of $m$ and then it is invertible.

\item We choose $\xi$  in the non-empty interval defined by \eqref{cond_xi_delta2} and then choose $K$ to satisfy  \eqref{cond_K_xi} with strict inequalities:
\begin{equation}\label{cond_K_xi-strict}
  \max\left\{\frac{\xi-1}{2}\,,\,{1}\right\}< K < \frac{\xi-1}{\delta_2}\,. %%%% \le deve cambiare in < per avere \mu<1 !!
\end{equation}
The strict inequality on the right of \eqref{cond_K_xi-strict} is needed both for the control on the number of interactions \cite[Lemma 6.2]{amadori-corli-siam} and for the decay of the reflected waves as the number of interactions increases, see \eqref{eq:mu} and Proposition \ref{prop:tilde-Fk}.
%and then $K$ in the non-empty intervals defined by \eqref{cond_xi_delta2}, \eqref{cond_K_xi}, respectively.
\item We choose $\knp$ so that \eqref{eq:knp} holds.
\end{enumerate}
We collect the results of the previous subsection into a single proposition.

\begin{proposition}[Local decreasing]\label{prop:last}
Consider the interaction of any two waves at time $t$. Let $m > 0$ be % either that $d(m)\le 1/3$ or
such that \eqref{ip_m_delta2} holds % (always true for $c(m)\le 1/3$)
and $C_o=C_o(m)$ as in Lemma \ref{lem:interazioni}.
If $\xi$, $K$, $\knp$ satisfy \eqref{cond_xi_delta2}, \eqref{cond_K_xi-strict} and \eqref{eq:knp}, respectively,
%\begin{equation}\label{cond_su_chi_xi}
%1+|\delta_2|<\xi<\frac{1}{\sqrt{d(m)}}\,,\quad
%\max\left\{1,\frac{\xi-1}{2}\right\}<K<\frac{\xi-1}{|\delta_2|}\,,\quad
%\knp<\frac{K}{C_o}\,,
%\end{equation}
then
\begin{equation}\label{eq:Fdecr}
\Delta F(t) \le 0\,.  %%% nota: vale < se K soddisfa <
\end{equation}
\end{proposition}

Now, we prove the global decreasing of $F$.

\begin{proposition}[Global decreasing] \label{global}
We choose parameters $m$, $\xi$, $K$, $\knp$ as in Proposition \ref{prop:last}.
Moreover, we assume that
\begin{equation}\label{eq:boundL}
\bar{L}(0+) \le m\hspace{0.8pt} c^2(m)
\end{equation}
and that the approximate solution is defined in $[0,T]$.
Then we have that $F(t)\le m$ and $\Delta F(t)\le0$ for every $t\in(0,T]$.
\end{proposition}

\begin{proof}
By Propositions \ref{Delta-F-2wave} and \ref{prop:DeltaF33} we know that $\Delta F\le 0$ if \eqref{rogna} holds.
By \eqref{eq:boundL} we deduce that $L(0+) \le m$ and by a recursion argument we find that for every $t\le T$
\begin{equation*}
F(t)\le F(0+) \le L(0+)(1 + K\delta_2) \le \xi^2 \bar{L}(0+)\le\frac{1}{c^2(m)}\bar{L}(0+)\le m\,.
\end{equation*}
This implies $\bar{L}(t)\le L(t)\le m$ for every $t\le T$ and in particular \eqref{rogna}.
\end{proof}

\section{The convergence and consistency of the algorithm}\label{sec:Cauchy}

\setcounter{equation}{0}

In this section we finally conclude the proof of Theorem \ref{thm:main}, focusing on the convergence and consistency of the front tracking algorithm.

For the algorithm to be well-defined, one has to verify that the total number of wave fronts and interactions is finite,
besides the fact that the size of rarefaction waves remains small. We already anticipated in the introduction that the algorithm used here
to construct the approximate solutions offers the advantage of getting quickly a bound on the total number of wave fronts.
As a matter of fact, at every interaction producing more than two outgoing waves the interaction potential $F$ decreases by a fixed positive
amount; hence, as in \cite[Lemma $6.2$]{amadori-corli-siam} one can prove that for large times any interaction involves only two incoming and two outgoing fronts.
The other two requirements are accomplished as in \cite[Proposition $6.3$]{amadori-corli-siam} and \cite[Lemma $6.1$]{amadori-corli-siam}, respectively.

The convergence follows from a standard application of Helly's Theorem, while for the consistency we need refined estimates to control the total size of the composite wave.

\subsection{Control of the total size of the composite wave}

The wave-front tracking scheme exploits the notion of generation order of a wave to prove that the strength
of the composite wave tends to zero as the approximation parameter $\nu$ tends to infinity:
this means that the $(2,0)$-wave becomes an entropic $2$-wave in the limit. More specifically,
for a physical wave $\gamma$ of family $1$ or $3$ we define its generation order $k_\gamma$ as in \cite[\S 6.2]{amadori-corli-siam};
on the other hand, for the $(2,0)$-wave we proceed as follows. We assign order $1$ to the $(2,0)$-wave generated at $t=0+$;
then, we keep its order unchanged in the cases where the Pseudo Accurate solver is used, while we set it to be equal to
$k_{\gamma}+1$ when the Pseudo Simplified solver is used with a  physical wave $\gamma$.

For any $k=1,2,\ldots$, we define
\begin{align*}
  L_k  &=  \sum_{\gamma>0\atop k_\gamma = k}|\gamma| + \xi
  \sum_{\gamma<0\atop k_\gamma = k}|\gamma| +
  \knp \,L_k^{0}\,,%\sum_{\gamma \in {\cal N\!P}\atop k_\gamma = k}|\gamma|\,,
  \\
  V_k & =  \sum_{\gamma>0,\,\mathcal{A}\atop k_\gamma = k}|\gamma| +
  \xi\sum_{\gamma<0,\,\mathcal{A}\atop k_\gamma = k}|\gamma|\,,\qquad Q_k\ =\ \delta_2V_k\,,
  \\
  F_k &= L_k + K\, Q_k\,,
\end{align*}
where $\gamma$ ranges over the set of $1$- and $3$-waves, as for \eqref{L-xi}.
Above we denoted
\begin{equation}\label{Lk0}
L_k^{0}=\sum_{\tau_k < t}|\eps_{2,0}-\delta_{2,0}|(\tau_k)\,,
\end{equation}
with $\tau_k$ denoting the interaction times where the outgoing composite wave has order of generation $k$.
As a consequence, only the times $\tau_k$ where the Pseudo Simplified solver is used give positive summands in \eqref{Lk0}:
when the Pseudo Accurate solver is used we have $\eps_{2,0} = \delta_{2,0}$.

For $k\in\naturali$, we introduce:
\begin{itemize}
\item $I_k = $ set of times when two waves $\alpha$, $\beta$ of same family interact,
with $\max\{k_\alpha, k_\beta\}=k$;
\item $J_k = $ set of times when a $1$- or a $3$-wave of order $k$ interacts with the $(2,0)$-wave.
\end{itemize}
We set   ${\cal T}_k = I_k\cup J_k$ and define
\begin{align}\label{eq:mu}
  \mu &\doteq \max\left\{\frac{1}{2K-1},\frac{\xi}{2K+1},\frac{K\delta_2+1}{\xi},\frac{\knp C_o}{K}\right\}\,.
\end{align}
We notice that $0<\mu<1$ by \eqref{cond_K_xi-strict} and \eqref{eq:knp}.

\begin{proposition}\label{lem:01}
  Let $m$, $\xi$, $K$ and $K_{np}$ satisfy the assumptions of Proposition \ref{prop:last} and assume  that $F(t)<m$ for all $t$.
   Then the following holds, for $\tau\in {\cal T}_h$, $h\ge 1$:
\begin{align}\label{Fk-segni-1}
&\Delta F_{h}<0\,,\qquad \Delta F_{h+1}>0\,,\\ \label{Fk-segni-2}
&\Delta F_k=0\qquad\,\, \, \hbox{ if }\ k\ge h+2\,.
\end{align}
Moreover,
\begin{equation}
\label{h=k-1b-XXX}
[\Delta F_{h+1}]_+ \le \mu \Bigl([\Delta F_h]_-   -
\sum_{\ell=1}^{h-1} \Delta F_\ell \Bigr)\,.
\end{equation}
\end{proposition}

\begin{remark}\label{ref:DeltaF<0} \rm Notice that Proposition~\ref{lem:01} let us improve Proposition~\ref{prop:last}. Indeed, recalling that ${\cal T}_h= I_{h}\cup J_h$, Proposition \ref{lem:01} implies, for $\tau \in I_{h}$,
\begin{equation*}%\label{Delta_W_piu_V}
\Delta F = \sum_{\ell=1}^{h-1}\Delta F_{\ell} \,-\, [\Delta F_h]_- \,+\, [\Delta F_{h+1}]_+  \le -(1-\mu)[\Delta F_h]_- < 0 \,,
\end{equation*}
while for $\tau \in J_{h}$, being  $\sum_{\ell=1}^{h-1} [\Delta F_\ell]_+=0$, it gives
\begin{equation*}%\label{Delta_W_piu_V}
\Delta F  \,=\, -  [\Delta F_h]_- \,+\, [\Delta F_{h+1}]_+ \le  -(1-\mu)[\Delta F_h]_- <0 \,.
\end{equation*}
Then, estimate \eqref{h=k-1b-XXX} quantifies the decrease in the functional $F$ and thus improves \eqref{eq:Fdecr}.
\end{remark}

\begin{proofof}{Proposition \ref{lem:01}} If $k\ge h+2$, no wave of order $k$ is involved and then \eqref{Fk-segni-2} holds\,.
To prove \eqref{Fk-segni-1} and \eqref{h=k-1b-XXX}, we distinguish between two cases.

\paragraph{\fbox{$\tau\in I_{h}$}} (Interactions between waves of $1$-, $3$-family).

Clearly the $F_k$'s do not vary when a $1$-wave interacts with a $3$-wave.
Then we consider interactions of waves of the same family, see Figure \ref{fig:inter_k-1}{\em (a)}.

Since $\tau\in I_{h}$, then $\Delta L_{h+1}>0$ and $0\le \Delta Q_{h+1}\le \delta_2\Delta L_{h+1}$.
Also, $\Delta F_{h}=\Delta L_{h} + K \Delta Q_{h}<0$, since both terms in the sum are negative or zero.
This proves \eqref{Fk-segni-1}.

By \eqref{Delta_L_xi_13} and \eqref{Delta_L_xi_13-SR} (see also \cite[(6.10)]{amadori-corli-siam}), we have that
\begin{equation}\label{0k-1}
[\Delta  L_{h+1}]_+ \, \le\,  \frac{1}{\xi} \Bigl([\Delta L_{h}]_- - \sum_{\ell = 1}^{h-1}\,\Delta L_\ell \Bigr).
\end{equation}
By (\ref{0k-1}), the estimate $0\le \Delta Q_{h+1}\le \delta_2\Delta L_{h+1}$ and \eqref{eq:mu} we deduce that
\begin{equation}
0<\Delta F_{h+1} \le (1+K\delta_2) [\Delta L_{h+1}]_+ \le \mu \biggl( [\Delta L_{h}]_- - \sum_{\ell = 1}^{h-1}\,\Delta L_\ell\biggr)\,.
\label{eq:ll}
\end{equation}
We now prove that
\begin{equation}
  [\Delta Q_{h}]_{-} - \sum_{\ell = 1}^{h-1}\Delta Q_\ell \ge 0\,,
\label{eq:mm}
\end{equation}
for which we only have to consider the case when $\Delta Q_\ell>0$ for an $\ell\le h-1$. In this case, $[\Delta Q_{h}]_- - \sum_{\ell = 1}^{h-1}\Delta Q_\ell
= -\delta_2\,\Delta V\ge \delta_2(-\Delta L + |\eps_1|)\ge 0$ because of \eqref{Delta_L_xi_13}, \eqref{Delta_L_xi_13-SR};
this proves \eqref{eq:mm}.
Therefore, for $\tau\in I_{h}$, estimate \eqref{h=k-1b-XXX} follows from (\ref{eq:ll}) and (\ref{eq:mm}).

%%%%%%%%%%%%%%%%%%%%%%%% inter31 %%%%%%%%%%%%%%%%%%%%%%%%%%%%%%
\begin{figure}[htbp]
\begin{picture}(100,80)(-100,-20)
\setlength{\unitlength}{0.7pt}

\put(60,0){
\put(0,40){\line(-1,-1){45}}
\put(-50,-15){\makebox(0,0){$\alpha_{h}$}}
\put(0,40){\line(-1,-3){15}}
\put(-15,-15){\makebox(0,0){$\beta_{\ell}$}}
\put(0,40){\line(1,2){20}}
\put(35,78){\makebox(0,0){$\eps_{3,\ell}$}}
\put(0,40){\line(-1,1){40}}
\put(-65,78){\makebox(0,0){$\eps_{1,h+1}$}}
\put(0,-35){\makebox(0,0){$(a)$}}
}

\put(320,0){
\put(0,40){\line(0,-1){50}}
\put(-25,-5){\makebox(0,0){$\delta_{(2,0),\ell}$}}
\put(0,40){\line(0,1){45}}
\put(-22,80){\makebox(0,0){$\delta_{(2,0),\ell}$}}
\put(0,40){\line(1,-2){25}}
\put(40,-10){\makebox(0,0){$\delta_{1,h}$}}
\put(0,40){\line(2,3){30}}
\put(55,78){\makebox(0,0){$\eps_{3,h+1}$}}
\put(0,40){\line(-2,1){60}}
\put(-65,78){\makebox(0,0){$\eps_{1,h}$}}
\put(0,-35){\makebox(0,0){$(b)$}}
}

\end{picture}

\caption{\label{fig:inter_k-1}{Interactions of waves; $h$ and $\ell$ denote generation orders. {\em (a)}: interaction of $3$-waves with $h\ge\ell$; {\em (b)}: interaction between a $1$-wave and the $(2,0)$-wave solved by the Pseudo Accurate solver.}}
\end{figure}
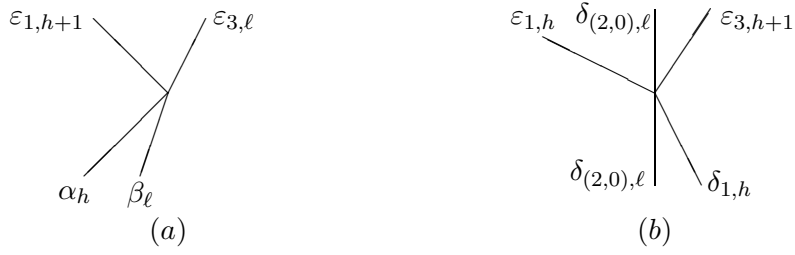
%%%%%%%%%%%%%%%%%%%%%%%%%%%%%%%%%%%%%%%%%%%%%%%%%%%%%%%%%%%%%%%

\paragraph{\fbox{$\tau\in J_{h}$}} (Interactions with the $(2,0)$-wave).

Since no wave of order $\le h-1$ interact, then (\ref{h=k-1b-XXX}) reduces to
\begin{equation}\label{h=k-1b-XXX-caso-J}
[\Delta F_{h+1}]_+ \le \mu [\Delta F_{h}]_- \,.
\end{equation}

To prove \eqref{h=k-1b-XXX-caso-J}, we first consider the case where the Pseudo Accurate solver is used, see Figure \ref{fig:inter_k-1}{\em (b)}.
Assume that a $1$-wave $\delta_1$ of order $h$ interacts with the $(2,0)$-wave. By \eqref{eq:patterns}, the reflected wave $\eps_3$ is of the
same type of the interacting wave and the transmitted one $\eps_1$. If $\delta_1>0$, then $\eps_1>0$ and $\eps_3>0$;
by Lemma \ref{lem:interazioni} this leads to
\begin{equation*}
\Delta F_{h} =\Delta L_{h}+K\Delta Q_{h}\le \frac{\delta_2|\delta_1|}{2}-K\delta_2|\delta_1|=-(2K-1) \frac{\delta_2|\delta_1|}{2}<0
\end{equation*}
by \eqref{cond_K_xi} and then, because of \eqref{eq:mu}, to
\begin{equation*}
[\Delta F_{h+1}]_+=\Delta L_{h+1} =|\eps_3|\le \frac{\delta_2|\delta_1|}{2}\le \frac{1}{2K-1}[\Delta F_{h}]_-\le \mu[\Delta F_{h}]_-\,.
\end{equation*}
The last estimate is also valid when $\delta_1<0$
(the only difference is that in the previous computations there is a factor $\xi$ both in $\Delta F_{h}$ and in $\Delta F_{h+1}$).

On the other hand, if we consider the interaction with a wave $\delta_3$ of order $h$ belonging to the third family, then the reflected
wave $\eps_1$ will be of a type different from that of $\delta_3$ and $\eps_3$. In this case, we first suppose $\delta_3,\eps_3>0$; then, $\eps_1<0$. As a consequence we have
\begin{equation*}
\Delta F_{h}=-|\eps_1|-K\delta_2|\delta_3|\le -(1+2K)|\eps_1|
\end{equation*}
and, therefore,
\begin{equation*}
[\Delta F_{h+1}]_+=\xi|\eps_1|=\frac{\xi}{1+2K}\left[(1+2K)|\eps_1|\right]\leq\frac{\xi}{1+2K}[\Delta F_{h}]_-\le \mu[\Delta F_{h}]_-\,,
\end{equation*}
because of \eqref{eq:mu}. In the other case, i.e. when $\delta_3,\eps_3<0$ and $\eps_1>0$, we have
\begin{equation*}
\Delta F_{h}  % =\xi(|\eps_3|-|\delta_3|)-K\xi\delta_2|\delta_3|
=-\xi|\eps_1| -K\xi \delta_2|\delta_3|%\leq-\xi|\eps_{1}|
\leq-\xi(1+2K)|\eps_1|
\end{equation*}
and
\begin{equation*}
[\Delta F_{h+1}]_+= %\Delta F_{h+1}=
|\eps_1| % =\frac{1}{\xi}(\xi|\eps_{1}|)\le \frac{1}{\xi}[\Delta F_{k-1}]_{-}
\le\frac{1}{\xi(1+2K)}[\Delta F_{h}]_{-}\le \mu[\Delta F_{h}]_-\,.
\end{equation*}

Now, we consider the case when the interacting wave has strength $|\delta|<\rho$ and then the Pseudo Simplified solver is used.
In this case a non-physical error of size $|\eps_{2,0}-\delta_{2,0}|$ and order $h+1$ appears. Thus, again by Lemma \ref{lem:interazioni},
\begin{equation*}
0< \Delta F_{h+1} = \knp \Delta L_{h+1} ^0 \le \knp C_o \delta_2 |\delta|,\qquad
\Delta L_{h}=0\,,\qquad \Delta Q_{h} \le  - \delta_2|\delta|\,.
\end{equation*}
Consequently, $[\Delta F_{h}]_- \ge K\delta_2|\delta|$ and
\begin{equation*}
[\Delta F_{h+1} ]_+  \le \frac{\knp C_o}{K}[\Delta F_{k-1}]_-\le \mu[\Delta F_{k-1}]_- \,.
\end{equation*}
Then (\ref{h=k-1b-XXX-caso-J}) is proved.
Finally we notice that, in all the above cases for $\tau\in J_{h}$, \eqref{Fk-segni-1} holds.
%By (\ref{h=k-1b-XXX-caso-J}) we get (\ref{h=k-1b-XXX}), since no wave of order $\le k-2$ interacts.
\end{proofof}

%\newpage
%Because of Proposition~\ref{lem:01}, we can conclude that for any $t\ge0$ and $k\ge 1$
%\begin{align}
%  \tilde L_k(t)\le \tilde F_k(t) & \le  \mu^{k-1}\cdot L(0) \cdot\left( 1
 %   + K \delta_2\right) \label{piu-bella}
%\end{align}
%in the same way as in \cite[Prop. 6.7]{amadori-corli-siam}.
Now, we proceed similarly as in \cite[Proposition 6.7]{amadori-corli-siam} to obtain a recursive estimate for $F_k$.
Indeed, the functional $F_k$ increases at times $\tau \in {\cal T}_{k-1}$, it decreases at $\tau \in {\cal T}_k$,
while it has not a definite sign for times $\tau\in {\cal T}_h$ with $h\ge k+1$.
For $F_1$ we have:
\begin{equation}\label{F1-0}
F_1(t)= F_1(0) - \sum_{{\cal T}_1} [\Delta F_1]_- + \sum_{h>1} \sum_{{\cal T}_h}  \Delta F_1\,,
\end{equation}
while for $F_k$ with $k\ge 2$ we use that $F_k(0)=0$ % and \eqref{h=k-1b-XXX} (without the last term because of $k=2$)
to obtain
\begin{equation}\label{Fk-0}
F_k(t)= \sum_{{\cal T}_{k-1}} [\Delta F_k]_+  -  \sum_{{\cal T}_k} [\Delta F_k]_- + \sum_{h>k} \sum_{{\cal T}_h}  \Delta F_k \,.
%\left([\Delta F_k]_+ - [\Delta F_k]_- \right)
\end{equation}
Here above we assumed that summations are done over interaction times $\tau<t$; the same notation is used in the following.
We consider now the last terms in \eqref{F1-0}, \eqref{Fk-0}:
\begin{equation*}
 \sum_{h>k} \sum_{{\cal T}_h}  \Delta F_k\,,\qquad k\ge 1\,.
\end{equation*}
The above contribution is different from zero (and then possibly positive) only if the interaction involves two waves of the same family,
one of order $k$ and the other of order $h$, with $h>k$. We denote by ${\cal T}_{h,k}$ the set of times at which an interaction of this type occurs.
Clearly ${\cal T}_{h,k}\subset {\cal T}_h$.

Moreover, we define the quantity
\begin{equation}\label{alpha-k}
\alpha_k(t)=  \sum_{\tau \in {\cal T}_{k-1},\tau<t} [\Delta F_k(\tau)]_+ \,,\qquad k\ge 2\,,
\end{equation}
that is, the first term on the right hand side of \eqref{Fk-0}.
Hence we rewrite \eqref{F1-0},  \eqref{Fk-0} as
\begin{align}\label{F1}
0\le F_1(t)&= F_1(0) - \sum_{{\cal T}_1} [\Delta F_1]_- + \sum_{ h> 1} \sum_{{\cal T}_{h,1}}  \Delta F_1\,,\\
\label{Fk}
0\le F_k(t)&= \alpha_k -  \sum_{{\cal T}_k} [\Delta F_k]_- + \sum_{h> k} \sum_{{\cal T}_{h,k}}  \Delta F_k\,, \qquad k\ge 2\,.
\end{align}

\begin{proposition}\label{prop:alpha} For $k\ge2$ one has
\begin{equation}\label{alpha-k-estimate}
\alpha_k \le \mu^{k-1} F_1(0) + \sum_{h\ge k }  \sum_{\ell=1}^{k-1} \sum_{{\cal T}_{h,\ell}}  \Delta F_\ell \,.
\end{equation}
\end{proposition}
\begin{proof} For $k=2$, we use \eqref{h=k-1b-XXX} and the positivity of $F_1$ to get
\begin{align*}
\alpha_2 &=  \sum_{{\cal T}_{1}} [\Delta F_2]_+ \,\le\, \mu \,\sum_{{\cal T}_{1}} [\Delta F_1]_-
\,\le\, \mu\left\{  F_1(0)  +   \sum_{h> 1} \sum_{{\cal T}_{h,1}} \Delta F_1 \right\}
\\&\le
 \mu F_1(0) +   \sum_{h\ge 2 } \sum_{{\cal T}_{h,1}} \Delta F_1\,,
\end{align*}
which is \eqref{alpha-k-estimate} for $k=2$.

By induction, assume that \eqref{alpha-k-estimate} holds for some $k\ge2$. Since $F_k\ge 0$, from \eqref{Fk} we get
\begin{equation*}
\sum_{{\cal T}_k} [\Delta F_k]_- \le  \alpha_k  \,+\, \sum_{h> k} \sum_{{\cal T}_{h,k}}  \Delta F_k\,.
\end{equation*}
Now, by definition \eqref{alpha-k}, by estimate \eqref{h=k-1b-XXX} and the previous inequality we find
\begin{align*}
\alpha_{k+1}= \sum_{{\cal T}_{k}} [\Delta F_{k+1}]_+ &\le \mu \sum_{{\cal T}_{k}} [\Delta F_{k}]_-
\,-\, \mu \sum_{\ell<k}\sum_{{\cal T}_{k,\ell}} \Delta F_\ell\\
&\le \mu \alpha_k \,+\, \mu  \sum_{h> k} \sum_{{\cal T}_{h,k}}  \Delta F_k \,-\, \mu \sum_{\ell<k}\sum_{{\cal T}_{k,\ell}} \Delta F_\ell\,.
\end{align*}
By using the induction hypothesis \eqref{alpha-k-estimate}, we get
\begin{equation*}
\alpha_{k+1} \le \mu^{k} F_1(0) + \mu \underbrace{\sum_{h,\ell \atop h\ge k>\ell } \sum_{{\cal T}_{h,\ell}}  \Delta F_\ell}_{(I)} %\\&&
\,+\, \mu  \sum_{h> k} \sum_{{\cal T}_{h,k}}  \Delta F_k \,-\, \mu \underbrace{\sum_{\ell<k}\sum_{{\cal T}_{k,\ell}} \Delta F_\ell}_{(I\!I)}\,.
\end{equation*}
Notice that
\begin{equation*}
(I)= (I\!I) + %\left\{ \sum_{k>\ell } \sum_{{\cal T}_{k,\ell}} +
\sum_{h,\ell \atop h> k>\ell } \sum_{{\cal T}_{h,\ell}}  %\right\}
 \Delta F_\ell\,,
\end{equation*}
so that
\begin{align*}
\alpha_{k+1} &\le \mu^{k} F_1(0) + \mu \sum_{h,\ell \atop h> k>\ell } \sum_{{\cal T}_{h,\ell}}  \Delta F_\ell\,+\,
 \mu  \sum_{h> k} \sum_{{\cal T}_{h,k}}  \Delta F_k \\
&= \mu^{k} F_1(0) + \mu \sum_{h,\ell \atop h> k\ge \ell } \sum_{{\cal T}_{h,\ell}}  \Delta F_\ell
\end{align*}
from which we deduce \eqref{alpha-k-estimate} for $k+1$, since $\mu<1$.
\end{proof}

\begin{proposition}\label{prop:tilde-Fk} For $k\ge2$ one has
\begin{equation}\label{estimate-tildeFj}
\tilde F_k(t) \ \dot =\  \sum_{j\ge k} F_j(t) \le {\mu^{k-1}} F_1(0)\,.
\end{equation}
\end{proposition}
\begin{proof}
For $k\ge 2$ we have $\tilde F_k(0)=0$. Moreover, we also deduce:
\begin{itemize}
\item $\Delta \tilde F_k(\tau)=0$ for $\tau\in {\cal T}_h$, $h\le k-2$, by \eqref{Fk-segni-2};
\item $\Delta \tilde F_k(\tau)= \Delta F_k(\tau) >0$ for $\tau\in {\cal T}_{k-1}$, by \eqref{Fk-segni-1};
\item at last, for all $\tau\in {\cal T}_h$, $h\ge k$,
\begin{equation*}
\Delta \tilde F_k(\tau) \le - \sum_{\ell=1}^{k-1} \Delta F_\ell(\tau)\,,
\end{equation*}
by the property $\Delta F(\tau)<0$, see Remark \ref{ref:DeltaF<0}.
\end{itemize}
As a consequence of the above properties, using also \eqref{alpha-k} and \eqref{alpha-k-estimate}, we find
\begin{align*}
\tilde F_k(t) &= \alpha_k + \sum_{h\ge k}  \sum_{{\cal T}_{h}}  \Delta \tilde F_k \\
&\le  \mu^{k-1} F_1(0) + \sum_{h\ge k }  \sum_{\ell=1}^{k-1} \sum_{{\cal T}_{h,\ell}}  \Delta F_\ell -
 \sum_{h\ge k}  \sum_{\ell=1}^{k-1}  \sum_{{\cal T}_{h,\ell}} \Delta F_\ell \,=\,  \mu^{k-1} F_1(0)\,.
\end{align*}
\end{proof}

%\begin{remark}\rm
We can now proceed to determine parameters $\rho$ and $\eta$ as in \cite{amadori-corli-siam}.
Fix $\eta>0$ such that $\eta=\eta_\nu\rightarrow 0$ as $\nu\rightarrow \infty$ and estimate the total number of waves of order $<k$.
Then, for the strength of the composite wave it holds
\begin{align*}
|\gamma_{2,0}|(t)&\le \tilde{L}_k(t)+\sum_{h<k \atop\tau_h<t}|\eps_{2,0}-\delta_{2,0}|(\tau_h)\le \\&\le   \mu^{k-1}\cdot L(0) \cdot\left( 1
    + K \delta_2\right) + C_o\rho\, \delta_2\, [\text{number of fronts of order $<k$}]<\frac{1}{\nu}\,,
\end{align*}
by choosing $k$ sufficiently large to have the first term $\le 1/(2\nu)$ and, then, $\rho=\rho_{\nu}$ small enough to have
the second term also $\le 1/(2\nu)$.
%\end{remark}

\begin{remark}\rm\label{rem:65}
Proposition \ref{prop:alpha} improves Lemma 6.6 in \cite{amadori-corli-siam}, because of $\Delta F_\ell$ on the right hand side of
\eqref{alpha-k-estimate} in place of $[\Delta F_\ell]_+$.  This is obtained under the same local interaction estimates \eqref{Fk-segni-1}--\eqref{h=k-1b-XXX}. Moreover, Proposition~\ref{prop:tilde-Fk} is only based on
 Proposition \ref{prop:alpha} and on $\Delta F<0$. Hence the same argument could be applied to the general case treated in  \cite{amadori-corli-siam}, and improve the related result by avoiding some technical assumptions due to the presence of non-physical waves.
\end{remark}

%\begin{remark}\rm
%%From the above proof we see that $\Delta L\le 0$ for $\xi=1$.
%%This was a key point in \cite{Nishida68}, where however a
%%different choice of strengths was done. In \cite{AmadoriGuerra01}
%%the inequality $\Delta L\le 0$ was proved to hold also for
%%$1<\xi\leq \xi_o$, for some $\xi_o>1$; the condition
%%(\ref{eq:sogliazza})${}_1$ gives an estimate of such a threshold.
%
%%More precisely
%  Recall that, in the first two cases of Proposition
%  \ref{prop:DeltaF33} we have $\Delta L\le 0$ for every
%  $\xi\ge1$. The third case is analyzed in detail in Lemma
%  \ref{lem:second}; we prove there that $\Delta L\le 0$ for any
%  $\xi > 1$ if $c(m) \le 1/2$, while we need $1<\xi\le
%  \frac{1}{2c(m)-1}$ if $c(m)>1/2$.
%\end{remark}
%%%%%%%%%%%%%%%%%%%%%%%%%%%%%%%%%%%

%\newpage
\subsection{Proof of Theorem \ref{thm:main} and a comparison}\label{subsec:proof}

In this last section we accomplish the proof of Theorem \ref{thm:main} and compare the result we obtain with that proved in \cite{amadori-corli-siam,amadori-corli-source}.

\begin{proofof}{Theorem \ref{thm:main}} It only remains to reinterpret the choice of the parameter $m$
in terms of the assumption~\eqref{hyp2} on the initial data.
Recalling Proposition~\ref{global}, \eqref{ip_m_delta2} and since
\begin{equation*}
\bar{L}(0+)\le\frac{1}{2}\tv\left(\log(p_o)\right) + \frac{1}{2\inf a_o}\tv(u_o)\,,  %<  m \hspace{0.7pt} d(m)
\end{equation*}
we look for $m$ satisfying
\begin{align}
    |\delta_2|<\frac{1}{{c(m)}}-1 \,=\, \frac{2}{\cosh m -1} &\, \dot = \, w(m)\,,\label{hyp1}
    \\
    \tv\left(\log(p_o)\right) \,+\, \frac{1}{\min\{a_r,a_\ell\}}%{\inf a_o}
    \tv(u_o)  < 2m\hspace{.8pt} c^2(m) &\, \dot =\, z(m)\,. \label{hyp2-1}
\end{align}
Notice that $w(m)$ is strictly decreasing from $\reali_+$ to $\reali_+$, while $z(m)$ is strictly increasing on the same sets.
Since $|\delta_2|<2$, we restrict the choice of the parameter to have $w(m)\in (0,2)$, that is $\cosh m>2$ and then
\begin{equation*}
m>\bar m = \cosh^{-1}(2) =  \log\left(2+\sqrt 3\right) \,.%  \qquad w(\bar m) =2\qquad (\mbox{equivalently }\ \cosh \bar m=2).% c(\bar m) = 1/3).
\end{equation*}
We can now define
\begin{equation}\label{K}
K(r)\,\dot =\, z\left(w^{-1}(r) \right)\,,\qquad r\in (0,2)\,,
\end{equation}
which can be written explicitly as
\begin{equation}
K(r)=\frac{2}{(1+r)^2}\,c^{-1}\left(\frac{1}{1+r}\right)
=\frac{2}{(1+r)^2} \log\left(\frac{2}{r}+1+\frac{2}{r}\sqrt{1+r}\right)\,.
\label{eq:K-explicit}
\end{equation}
It is easy to check that $K$ satisfies properties \eqref{rage}.

Hence, if the assumption \eqref{hyp2} holds, namely
\begin{equation*}
  \tv\left(\log(p_o)\right) \,+\, \frac{1}{\min\{a_r,a_\ell\}}  \tv(u_o) < K(|\delta_2|) \,,
\end{equation*}
it is easy to prove that one can choose $m>\bar m$ such that \eqref{hyp1}, \eqref{hyp2-1} hold. Finally, in order to pass to the limit and prove the convergence to a weak solution, one can proceed as in \cite{Bressanbook}. Theorem \ref{thm:main} is, therefore, completely proved.
\end{proofof}

Now, we make a comparison between Theorem~\ref{thm:main} and the main result in \cite{amadori-corli-siam}, which was proved to be equivalent to Theorem $3.1$ of \cite{amadori-corli-source}. Condition $(3.7)$ of the latter theorem, when applied to the current problem, can be written as
\begin{equation}\label{acs37}
\tv\left(\log(p_o)\right) + \frac{1}{\min\{a_r,a_\ell\}}%{\inf a_o}
\tv(u_o) < H(|\delta_2|)\,,
\end{equation}
where the function $H(r)$ is only defined for $r<1/2$ by
\begin{equation}\label{H}
H(r)\doteq 2(1-2r)k^{-1}(r)\,, \qquad k(m) = \frac{1-\sqrt{d(m)}}{2-\sqrt{d(m)}}\,.
\end{equation}
Here above, $d(m)$ is the damping coefficient introduced in \cite[Lemma 5.6]{amadori-corli-siam}, see Remark~\ref{rem-d(m)}.

%%%%%%%%%%%%%%%%%%%%%%%%%%%%%%%%%%%%

\begin{figure}[htbp]
\begin{center}
 %\hskip-5mm
 {\includegraphics[width=7cm]{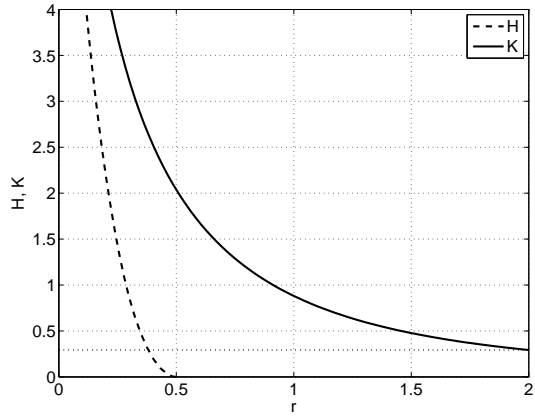}}
 %{\includegraphics[width=5cm]{}}
    %\begin{psfrags}
%    \psfrag{H}{$H$} \psfrag{A}{$A$} \psfrag{A_o}{$A_o$}
%    \includegraphics[width=5.0cm]{H}
%    \end{psfrags}
\end{center}
%\begin{picture}(100,0)(0,0)
%\setlength{\unitlength}{0.8pt}
%\put(0,0){\makebox{$(a)$}}\put(180,0){\makebox{$(b)$}}
%\end{picture}
\vspace{-6mm}
\caption{The functions $H$ (dashed line) and $K$ (solid line). The horizontal dotted line gives the asymptotic value $\frac29\log(2+\sqrt3)$ of $K$ for $r\to 2-$. }
\label{fig:HK}
\end{figure}

%%%%%%%%%%%%%%%%%%%%%%%%%%%%%%%%%%%%

Hence, the result of Theorem \ref{thm:main} is new for $1/2\le|\delta_2|<2$, including the case where the $2$-wave may be arbitrarily large, i.e. $|\delta_2|$ close to $2$. In order to compare \eqref{acs37} with \eqref{hyp2} in the common range $|\delta_2|<1/2$,
we set  $r=|\delta_2|\in(0,1/2)$ and rewrite $H$ as
\begin{align}
H(r)&=2(1-2r)\,d^{-1}\left(\bigl(\frac{1-2r}{1-r}\bigr)^2\right)\,.
\nonumber
\end{align}
Comparing this expression with \eqref{eq:K-explicit}, we notice that $1/(1+r)^2>(1-2r)$. Moreover, we have
\begin{equation*}
\frac{1}{1+r}> \bigl(\frac{1-2r}{1-r}\bigr)^2\,;
\end{equation*}
since $c<d$ and $c$ is strictly increasing, we have also that $c^{-1}\left(1/(1+r)\right)>k^{-1}(r)$. We deduce that $K(r)>H(r)$ for $0\le r <1/2$; see Figure~\ref{fig:HK}. Then, the conditions on the initial data obtained here considerably improve the ones required in the previous works \cite{amadori-corli-siam,amadori-corli-source}, albeit the latter were given for a more general case.

%%%%%%%%%%%%%%%%%%%%%%%%%%%%%%%%%%

\appendix
\section{Another interpretation of the damping coefficient $c$}
\setcounter{equation}{0}

The function $c$ introduced in \eqref{eq:chi_def} plays a fundamental role in controlling the size of the weight $\xi$ assigned to shock waves in the front-tracking scheme, see Proposition \ref{prop:DeltaF33}. In this appendix we show that the same coefficient $c$ also appears in the stability analysis of the Riemann problems of system \eqref{eq:system}, see \cite{Schochet-Glimm,AmCo12-Chambery}.

In \cite{Schochet-Glimm} Schochet proves that if the solution of a Riemann problem satisfies some {\em finiteness conditions} (also called $BV$-stability conditions), then {\em small} perturbations of bounded variation of its initial data give rise to a solution defined globally in time. The analysis for system \eqref{eq:system} was done in \cite{AmCo12-Chambery}, where it was proved that there are solutions to suitable Riemann problems that do not satisfy such conditions.

As in \cite[Lemma 1.2]{AmCo12-Chambery}, let us consider the pattern formed by a $1$-shock $\eps_1$, a $2$-wave $\eps_2$ and a $3$-shock $\eps_3$.
Maintaining the notation of that paper, we denote the states lying between waves with $U_0,U_1,U_2,U_3$, from left to right; see Figure \ref{fig:RP}.

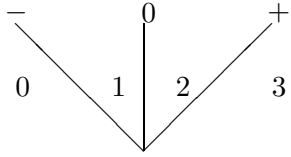
\begin{figure}[htbp]
\begin{picture}(100,50)(0,10)
\setlength{\unitlength}{0.8pt}
%axes:
\put(280,0){
\put(0,0){\line(-1,1){60}}
\put(0,0){\line(0,1){60}}
\put(0,0){\line(1,1){60}}

\put(-65,65){\makebox(0,0)[l]{$-$}}
\put(-1,65){\makebox(0,0)[l]{$0$}}
\put(58,65){\makebox(0,0)[l]{$+$}}

\put(-60,30){\makebox(0,0)[l]{$0$}}
\put(-15,30){\makebox(0,0)[l]{$1$}}
\put(15,30){\makebox(0,0)[l]{$2$}}
\put(60,30){\makebox(0,0)[l]{$3$}}
}
\end{picture}
\caption{States for the Riemann problem.}
\label{fig:RP}
\end{figure}

% We indicate with $c_1=a_1/v_1$, $c_2=a_2/v_2$ the characteristic speeds and with $s_-=-a_1/\sqrt{v_1v_0}$, $s_+=a_2/\sqrt{v_2v_3}$ the speeds of the shocks of the first and third family, respectively. Finally, we write $L_{\pm}$, $R_{\pm}$ for the left and
% right eigenvectors of the first and third family, while we use $[U]_{\pm}$ to indicate the variation of $U$ along the $1$- and $3$-shock. Then, let us introduce the following quantities
We use $c_1=a_1/v_1$, $c_2=a_2/v_2$ to indicate the characteristic speeds and 
$s_-=-a_1/\sqrt{v_1v_0}$, $s_+=a_2/\sqrt{v_2v_3}$ to indicate the speeds of the 
shocks of the first and third family, respectively. Finally, we write $L_{\pm}$, 
$R_{\pm}$ for the left and right eigenvectors of the first and third family, 
while we let $[U]_{\pm}$ be the variation of $U$ along the $1$- and $3$-shock. 
Then, let us introduce the following quantities
\[
A=|R^{(-)}|=\left|\frac{c_1+s_-}{c_1-s_-}\cdot\frac{L_+(U_1)\cdot[U]_-}{L_-(U_1)\cdot[U]_-}\right|\,,\quad B=|R^{(+)}|=\left|\frac{c_2-s_+}{c_2+s_+}\cdot\frac{L_-(U_2)\cdot[U]_+}{L_+(U_2)\cdot [U]_+}\right|\,,
\]
which represent some coefficients of the reflection matrices $R^{(-)}_{>,\le}$ 
and $R^{(+)}_{<,\ge}$ appearing in \cite{AmCo12-Chambery}.

\begin{lemma}\label{lem:Schochet}
Under the notation in \eqref{eq:chi_def}, we have $A=c(\eps_1)$ and $B=c(\eps_3)$.
\end{lemma}
\begin{proof} First, notice that
\[
\frac{L_+(U_1)\cdot[U]_-}{L_-(U_1)\cdot[U]_-}=\frac{-c_1(v_1-v_0)+(u_1-u_0)}{c_1(v_1-v_0)+(u_1-u_0)}
\]
and, recalling that along a shock of the first family it holds $u_1-u_0=-s_-(v_1-v_0)$, the previous quantity becomes $(-c_1-s_-)/(c_1-s_-)$. Therefore,
\[
A=\bigl(\frac{c_1+s_-}{c_1-s_-}\bigr)^2
% =\left(\frac{1/{v_1}-1/\sqrt{v_1v_0}}{1/{v_1}+1/{\sqrt{v_1v_0}}}\right)^2
=\left(\frac{{v_0}/{v_1}-\sqrt{{v_0}/{v_1}}}{{v_0}/{v_1}+\sqrt{{v_0}/{v_1}}}\right)^2.
\]
By definition \eqref{eq:strengths}, we get ${v_0}/{v_1}=\exp(-2\eps_1)$ and, finally, we find
\begin{align*}
A&=%\bigl(\frac{\exp(-2\eps_1)-\exp(-\eps_1)}{\exp(-2\eps_1)+\exp(-\eps_1)}\bigr)^2=
\bigl(\frac{\exp(-\eps_1/2)-\exp(\eps_1/2)}{\exp(-\eps_1/2)+\exp(\eps_1/2)}\bigr)^2=%\bigl(\frac{\sinh(\eps_1/2)}{\cosh(\eps_1/2)}\bigr)^2=
\tanh^2(\eps_1/2)=\frac{\cosh(\eps_1)-1}{\cosh(\eps_1)+1}=c(\eps_1)\,.
\end{align*}
By similar computations we get also
$B=(\cosh(\eps_3)-1)/(\cosh(\eps_3)+1)=c(\eps_3)$.
\end{proof}

By Lemma \ref{lem:Schochet}, the finiteness condition of \cite{Schochet-Glimm} for the above pattern of two shock waves and the contact discontinuity can be written as
\begin{equation}\label{eq:Schochet-stable}
c(\eps_1)c(\eps_3)\eps_2^2 - \left(c(\eps_1)+c(\eps_3)\right)|\eps_2| + 2\left(1-c(\eps_1)c(\eps_3)\right)>0\,.
\end{equation}
This condition makes explicit the analogous one provided in \cite[(14)]{AmCo12-Chambery}. We remark that condition \eqref{eq:Schochet-stable} is satisfied for {\em every} shock $\eps_3$ (for example) if it holds in the degenerate case $c(\eps_3)=1$, \cite{AmCo12-Chambery}; in such a case, it simply reduces to
\[
1+|\eps_2| \le \frac{1}{c(\eps_1)},
\]
which reminds of \eqref{cond_xi_delta2}.

{\small
%\bibliography{refe}
\bibliographystyle{abbrv}

}
\end{document}